\documentclass[preprint]{imsart}

\usepackage{amsmath,amssymb,amsthm,amsfonts,amstext,amsbsy,amscd}
\usepackage{mathrsfs}
\usepackage{mathabx}

\usepackage{float}
\newcommand{\nrm}[1]{\lVert #1 \rVert}
\usepackage{multirow}
\usepackage{multicol}
\usepackage{bm}
\usepackage{subcaption}
\usepackage{multirow}
\usepackage{caption}
\usepackage{tabularx}
\usepackage{booktabs}
\usepackage{multirow}
\usepackage{graphicx}
\usepackage{multicol}
\usepackage{multirow}
\usepackage{enumerate}
\usepackage{latexsym}

\usepackage[utf8]{inputenc}
\usepackage{dsfont}
\usepackage{color}
\usepackage{tikz}
\RequirePackage[numbers]{natbib}
\RequirePackage[colorlinks,citecolor=blue,urlcolor=blue]{hyperref}

\startlocaldefs
\theoremstyle{plain}
\newtheorem{theorem}{Theorem}
\newtheorem{lemma}{Lemma}

\newtheorem{proposition}{Proposition}
\theoremstyle{remark}

\newtheorem{example}{Example}
\newtheorem{definition}{Definition}

\newcommand{\E}{\mathbb{E}}

\newcommand{\1}{\mathbf{1}}
\newcommand{\R}{\mathbb{R}}
\newcommand{\N}{\mathbb{N}}

\newcommand{\eps}{\varepsilon}

\renewcommand{\tilde}{\widetilde}
\renewcommand{\hat}{\widehat}
\renewcommand{\check}{\widecheck}

\definecolor{C}{RGB}{ 19, 141, 155 }

\newcommand{\ind}{\mathbf{1}}

\endlocaldefs

\begin{document}

\begin{frontmatter}
\title{Adaptive minimax estimation for discretely observed Lévy processes}
\runtitle{Adaptive minimax estimation for discretely observed Lévy processes}

\begin{aug}
\author[A]{\fnms{Céline}~\snm{Duval} \ead[label=e1]{celine.duval@sorbonne-universite.fr}},
\author[B]{\fnms{Taher}~\snm{Jalal}\ead[label=e2]{taher.jalal@uvsq.fr}}
\and
\author[B]{\fnms{Ester}~\snm{Mariucci}\ead[label=e3]{ester.mariucci@uvsq.fr}}
\address[A]{Laboratoire de Probabilités, Statistique et Modélisation, Sorbonne Université\printead[presep={,\ }]{e1}}

\address[B]{Laboratoire de Math\'ematiques de Versailles, UVSQ - Universit\'e Paris-Saclay\printead[presep={,\ }]{e2,e3}}
\end{aug}

\begin{abstract}
In this paper, we study the nonparametric estimation of the density \( f_\Delta \) of an increment of a Lévy process \( X \) based on \( n \) observations with a sampling rate \( \Delta \). The class of Lévy processes considered is broad, including both processes with a Gaussian component and pure jump processes. A key focus is on processes where \( f_\Delta \) is smooth for all \( \Delta \). We introduce a spectral estimator of \( f_\Delta \) and derive both upper and lower bounds, showing that the estimator is minimax optimal in both low- and high-frequency regimes.
Our results differ from existing work by offering weaker, easily verifiable assumptions and providing non-asymptotic results that explicitly depend on \( \Delta \). In low-frequency settings, we recover parametric convergence rates, while in high-frequency settings, we identify two regimes based on whether the Gaussian or jump components dominate. The rates of convergence are closely tied to the jump activity, with continuity between the Gaussian case and more general jump processes. Additionally, we propose a fully data-driven estimator with proven simplicity and rapid implementation, supported by numerical experiments.
\end{abstract}

\begin{keyword}
\kwd{Lévy processes}
\kwd{Density estimation}
\kwd{Spectral estimator}
\kwd{Small jumps}
\kwd{Infinitely divisible distributions}
\kwd{Minimax rates of convergence}
\end{keyword}

\end{frontmatter}

\section{Introduction}

\subsection{Motivations}

Lévy processes are among the most mathematically tractable examples of jump processes, they have been extensively studied and widely used in the mathematical modeling of phenomena that may exhibit abrupt changes, and therefore are matters of interest in various fields such as mathematical finance, seismology, climatology, neuroscience (see e.g. \cite{barndorff2012levy,noven2015levy} for reviews and other applications). Very recently, they have gained renewed visibility due to the increasing interest in machine learning for heavy-tailed distributions and the link between stochastic gradient descent (SGD) and SDE driven by pure jump Lévy processes of infinite activity (see e.g. \cite{gurbuzbalaban2021heavy}).
Their structure is rather rigid  (their increments are stationary and independent) but they have often been used as proxies to establish results for jump processes with a more flexible structure, e.g., Itô's semi-martingales. 

Inference for Lévy processes often focuses on estimating the Lévy triplet (see Section \ref{sec:not}) and in particular the Lévy measure that governs jump dynamics (see for instance \cite{bucher2013nonparametric,MR2816339,duval2021spectral,figueroa2009nonparametric,MR2683463,MR2546805}) or the Blumenthal-Getoor index (see for instance \cite{10.1214/08-AOS640} and \cite{Mies2019RateoptimalEO}). 
Estimating the law of an increment of a Lévy process in a nonparametric manner have been less studied, though it provides critical insights into the underlying process, allowing for more refined model selection to fit a parametric model. Unlike parametric methods, which assume a specific distributional form (e.g., Poisson or stable jumps), nonparametric estimation makes fewer assumptions, allowing the data to guide the shape of the distribution. This is particularly important when dealing with real-world data, where the true distribution of the process may not fit common parametric models. By non-parametrically estimating the law of the increments, we can then capture complex behaviors such as heavy tails or asymmetric jumps.  In this paper, we focus on the nonparametric estimation of the density $f_{\Delta}$ of an increment of a Lévy process $X$ from $n$ observations collected with a sampling rate $\Delta$, i.e. $f_{\Delta}$ is the density of $X_\Delta$. Given the Lévy structure of the process, the corresponding statistical model is associated to a family of probabilities induced by independent and identically distributed (i.i.d.) observations. While this problem may initially appear similar to classical density estimation, the dependence of the i.i.d. random variables on time introduces additional complexity. Moreover, studying the smoothness of $f_{\Delta}$ in relation to the Lévy triplet is a task specific to this context.
We focus on Lévy processes for which the density \( f_\Delta \) exists and does not depend on \( \Delta \). A key contribution of this paper is to address a large class of Lévy processes with conditions that are easy to verify based on the Lévy triplet. Broadly speaking, this class coincides with Lévy processes that have a positive Blumenthal-Getoor index. We provide general conditions on the Lévy triplet that ensure that \( f_{\Delta} \) is smooth for all \( \Delta \). While this smoothness is typically guaranteed when a non-zero Gaussian component is present, it can also occur in pure jump processes whose Lévy density is bounded from below near zero, as stated in Assumption \eqref{Ass:p} below. The latter ensures that the Lévy process has a positive Blumenthal-Getoor index, thereby allowing for the regularizing effect of small jumps having infinite activity to aid in recovering the density \( f_\Delta \).

In this paper, we study a spectral estimator of  \( f_\Delta \) and derive both upper and lower bounds to assess its performance. We prove that our estimator is minimax optimal in both low and high-frequency observation regimes, see Theorems \ref{rateUP} and \ref{thmLB}. 
The work closest to this context is Section 4 in \cite{kappus2015nonparametric}, which considers the low-frequency setting where $\Delta=1$ is fixed. However, the proposed estimator is more complex, and the convergence rates are derived for generic regularity classes of $f_1$, without clarifying the connection between the Lévy triplet and the regularity of $f_1$ (see, e.g., Theorem 4.1 in \cite{kappus2015nonparametric}). In contrast, our work introduces much weaker and easily verifiable assumptions, yielding non-asymptotic results that rigorously assess how the risk of estimation varies with the sampling rate \( \Delta \). In low-frequency settings, we observe a parametric rate of convergence. As \( \Delta \) approaches zero, the situation becomes more intricate, revealing two distinct regimes depending on whether the Gaussian component dominates the jump part. If the Gaussian component prevails, the rate for the squared integrated \( L^2 \) risk is of order \( 1/(n\Delta^{1/2}) \); if not, the rate is of order \( 1/(n\Delta^{1/\alpha}) \), where \( \alpha \) can be interpreted as the Blumenthal-Getoor index of the process (see Proposition \ref{propBG} and Example \ref{ex:BGnot}). Notably, there is continuity with the Gaussian case, which roughly corresponds to the case \( \alpha = 2 \). In high-frequency settings, for pure jump Lévy processes, the more frequent the small jumps of the process, the easier it becomes to estimate the density. In any case, the convergence rate is explicitly tied to the jump activity of the process. The results from Theorem \ref{thm:rate_sigma} further quantify the conditions under which the continuous part dominates the jump part, depending on the diffusion coefficient and the sampling rate. The results for $\Delta\to 0$ are entirely new; they are minimax optimal, and the proof of the lower bound presents significant challenges. Notably, the techniques employed in our proofs are non-classical and original, leveraging the Local Asymptotic Normality property of a class of stable processes. This  approach may also prove useful for establishing lower bounds in other contexts within the field.

The rates of convergence are derived using theoretical cutoffs that balance variance and squared bias. However, these optimal cutoffs depend on unknown parameters from Assumption \eqref{Ass:p}. To address this, we adapt a novel adaptive estimator, inspired by the work in \cite{ammous2024adaptive}, which is applied here in the context of estimating a time-indexed density. We propose this estimator not only for its theoretical robustness but also for its simplicity and ease of implementation. Theorem \ref{thm:adapt} ensures that this fully data driven procedure is minimax optimal. Moreover, its performance has been validated through numerical experiments, summarized in Section \ref{sec:numerics}. In the remainder of this section, the principal notations and definitions are introduced.

\subsection{Setting and notations\label{sec:not}}
 Consider a  L\'evy process $X$  characterized  by its L\'evy triplet $(b,\sigma^{2},\nu)$ where $b\in \R$, $\sigma\ge0$, and  $\nu$ is a Borel measure on $\R$ such that
\begin{align*}
\nu(\{0\})=0,\quad \nu(\R)=\infty\quad \textnormal{ and }\quad  \int_{\R}(x^2\wedge 1)\nu(d x)<\infty.
\end{align*}  
The Lévy-Khintchine formula gives the characteristic function of $X$ at any time $t\ge 0$:
\begin{align}\label{LK}\phi_t(u):=\E[e^{iuX_{t}}]=\exp\bigg(it u b-\frac{t\sigma^{2}u^{2}}{2}+t\int_{\R}\big(e^{iux}-1-iux\1_{|x|<1}\big)\nu(dx)\bigg),\quad u\in\R.
\end{align} 
 Thanks to the L\'evy-It\^o decomposition, $X$ can be written as
\begin{align}
X_t=tb_{\nu}+\sigma W_{t}+X_t^S+X_t^B,\label{eq:Levy}
\end{align}
where $W$ is a standard Brownian motion, \begin{equation*}\label{eq:bnueps}b_{\nu}:=\begin{cases}
           b+  \int_{|x|\leq 1} x\nu(dx)\quad &\text{if}\quad \int_{|x|\leq 1}|x|\nu(dx)<\infty,\\
            b \quad &\text{if}\quad \int_{|x|\leq 1}|x|\nu(dx)=\infty,
             \end{cases}   
\end{equation*} $X^B$ is a compound Poisson process independent of $X^S$ with Lévy measure $\nu\mathbf1_{[-1,1]^{c}}$, $X^S$ is a centered martingale accounting for the jumps of $X$ of size smaller than $1$, i.e. $$X^{S}_{t}=\lim_{\eta\to 0} \bigg(\sum_{s\leq t} \Delta X_s\1_{\eta<|\Delta X_s|\leq 1}-t\int_{\eta<|x|\leq 1} x\nu(dx)\bigg),$$ where $\Delta X_r$ denotes the jump at time $r$ of the c\`adl\`ag process $X$: $\Delta X_r=X_r-\lim_{s\uparrow r} X_s.$ The Lévy processes $W$, $X^{S}$ and $X^{B}$ in \eqref{eq:Levy} are independent.  We refer to \cite{Bertoin} for an overview of the main properties of Lévy processes, including a thorough discussion of the Lévy-Khintchine formula and the Lévy-Itô decomposition.

Consider the i.i.d. observations $\mathbb X=(X_{i\Delta}-X_{(i-1)\Delta})_{i=1}^n$ with $X_0=0$ a.s.. 
Our aim is to estimate the density $f_{\Delta}$ of $X_\Delta$ from $\mathbb X$ both under the assumption $\Delta> 0$ fixed and $\Delta\to 0$, and compute the $L^2$ integrated risk. 
The estimation strategy that we analyse is based on a spectral approach, and we  use the following notations. Given a random variable $Z$, 
$\phi_Z(u)=\E[e^{iu Z} ]$ denotes the characteristic function of $Z$.  For $f\in L^1(\R)$, $\mathcal F f(u) =\int e^{iux } f(x) d x$ is the Fourier transform. Moreover, we denote by $\| \cdot  \|$ the $L^2$-norm of functions, $\| f\|^2:= \int |f(x)|^2 d x$. Given some function $f\in L^1(\R) \cap L^2(\R)$,  we denote by $f_m$, $m>0$, the uniquely defined function with Fourier transform $\mathcal F f_m= (\mathcal F f)\mathds{1}_{[-m,m]}$.

\section{Main results}\label{main}

\subsection{Class of Lévy processes}

We are interested in estimating the density  with respect to the Lebesgue measure of an increment $X_{\Delta}$. Therefore, we first give conditions on the Lévy triplet ensuring that $X_{t}$ admits a Lebesgue density for any ${t}>0.$ It is systematically the case in presence of a nonzero Gaussian component, $\sigma>0$. In the case  $\sigma=0$, a sufficient condition derived from Theorem 27.7 of \cite{sato} (see also  \cite{tucker_continuous_1964} and \cite{fisz_condition_1963}) is that $\nu(\R)=\infty$ and $\nu$ is absolutely continuous. Each of these conditions alone is not sufficient. Indeed, if $\nu(\R)<\infty$, $X$ is a compound Poisson process and its increment $X_{t}$ has a mass at 0, for all $t>0$. If $\nu(\R)=\infty$ but $\nu$ is not absolutely continuous a Lévy process such that $X_{t},\ t>0,$ does not have a Lebesgue density is constructed in \cite{orey_continuity_1968}.

In the study of the risk, the regularity of the density of $X_{t}$ plays a central role. We introduce constraints on the Lévy triplet ensuring that the density of $X_{t}$, $t>0$, is smooth for any $t>0$. 
\begin{definition}\label{def:T}
Define  $\mathcal T _{+}$ and $\mathcal T _{M,\alpha}$ the following classes of Lévy triplets:
 \begin{align*}\begin{cases}
\mathcal T _{+} &:= \left \{(b,\sigma^{2},\nu),\ b\in\R,\ \sigma^{2}>0,\ \nu \mbox{ is a Lévy measure}\right\}\\
\mathcal T _{M,\alpha} &:=\left\{(b,\sigma^{2},\nu),\ b\in\R, \ \sigma^2\geq 0,\ \nu \mbox{ admits a Lévy density $ p$ satisfying \eqref{Ass:p}}\right\}
\end{cases}\end{align*}
where for $0<\alpha< 2$ and $M>0$, $p$ satisfies \eqref{Ass:p} if:
\begin{align}\label{Ass:p} 
\int_{[-\eta,\eta]}x^{2}p(x)dx\geq {M}{\eta^{2-\alpha}}, \quad \forall0<\eta\leq 1
\tag{$\mathcal A_{M,\alpha}$}.
\end{align}
Set $\mathcal T _{+, M,\alpha} = \mathcal T_+ \cap \mathcal T _{M,\alpha}$ and  $\mathcal T _{0, M,\alpha} = \mathcal T _{M,\alpha} \setminus \mathcal T _{+}$.
Denote by $\mathcal F_+^\Delta$ (resp. $\mathcal F_{0,M,\alpha} ^\Delta$) the set of densities of  $X_\Delta$ for any $X$ that belongs to $\mathcal T _{+}$ (resp. $\mathcal T _{0,M,\alpha}$).
\end{definition}

Assumption \eqref{Ass:p} implies that the  Lévy process has infinite activity, $\nu(\R)=\infty$, excluding compound Poisson processes.  Any Lévy process whose triplet lies in $\mathcal T _{+}$ or $\mathcal{T}_{0,M,\alpha}$ is such that $X_{t}$ admits a Lebesgue density $f_{t},$ for any $t>0$. Moreover, $f_{t}$  is in $L^{2}(\R)$ and is infinitely differentiable for any $t>0$.
This can be checked using the convolution structure in \eqref{eq:Levy} and Young's inequality. Either $(b,\sigma^{2},\nu)\in \mathcal T_{+}$ and we use that Gaussian densities are infinitely differentiable or   $(b,\sigma^{2},\nu)\in \mathcal T_{0,M,\alpha}$  and we rely on  Lemma \ref{lemmapicard} below. More generally, necessary and sufficient conditions for the existence and smoothness of the density of $X_t$ can be found in, e.g., \cite{knopova2013note}.

Recall  the definition of the Blumenthal-Getoor index of $X$: $$\beta:=\inf\bigg\{\gamma>0,\ \int_{[-1,1]}|x|^{\gamma}\nu(dx)<\infty\bigg\}.$$ Observe that if the Lévy density $X$ satisfies \eqref{Ass:p}, then $\alpha\le \beta$ (see Proposition \ref{propBG}). 
Assumption \eqref{Ass:p}  excludes pure jump Lévy processes with null Blumenthal-Getoor index such as the Gamma process with Lévy density $p(x)=\frac{e^{-x}}x\mathbf{1}_{x>0}$.  An increment $X_{t}$ of the latter follows a $\Gamma(t,1)$ distribution which admits a density $f_{t}(x)=\frac{1}{\Gamma(t)}x^{t-1}e^{-x}\mathbf{1}_{x>0}$ whose regularity depends on $t$: it is continuous for any $t>1$ and it belongs to any Sobolev space with regularity $\beta$ for any $\beta < t- \frac12$.  It is out of the scope of the present work.

\subsection{Estimation and associated rates}

Let $\Delta>0$,  
consider for all $u\in \R$ the empirical characteristic function
\begin{align}\label{eq:phiZ}
\hat \phi_{X_{\Delta}}(u) =\frac{ 1}n\sum_{j=1}^{n}{e^{iu(X_{j\Delta}-X_{(j-1)\Delta})}} .
\end{align} 
Let $m>0$, from \eqref{eq:phiZ} we derive an estimator of $f_{\Delta}$,  using a spectral cut-off as the latter quantity may not be in $L^{1}(\R)$: 
\begin{align}
\label{eq:estFg}\hat f_{{\Delta},m}(x) =\frac{ 1}{2\pi}\int_{-m}^{m}\hat \phi_{X_{\Delta}}(u)e^{-iux}du.
\end{align} 
The following result gives an upper bound for the integrated $L^{2}$-risk of $\hat f_{{\Delta},m}$. Proof of Theorem \ref{thm1} is omitted as it relies on classical computations and on the use of the Plancherel equality.
\begin{theorem}\label{thm1}
Let $X$ be a Lévy process such that $X_\Delta$ admits a density $f_\Delta$, $\Delta>0.$ The estimator $\hat f_{\Delta,m}$ defined in \eqref{eq:estFg} satisfies, for all $m\geq 0$:
$$\E[\|\hat{f}_{\Delta,m}-f_{\Delta}\|^2]\leq   \|f_{\Delta,m}-f_{\Delta}\|^2+\frac{1}{\pi} \frac mn . $$ 
\end{theorem}
Rates of convergence can be computed using that  the bias term is:
\begin{align}\label{biais}
2\pi \|f_{\Delta,m}-f_{\Delta}\|^{2}
 &=\int_{|u|\ge m}\exp\left(-\Delta\sigma^{2}u^{2}+2\Delta\int_{\R}(\cos(ux)-1)\nu(dx)\right)du.
 \end{align}
Clearly, the order of the bias term heavily depends on the presence of a Brownian component and/or the activity of the jump part of $X$. 
The results on the convergence rate of $\hat{f}_{\Delta,m}$ are collected in the following statement whose  proof is postponed to Section \ref{sec:proofs}.
\begin{theorem}\label{rateUP}
Let $X$ be a Lévy process of Lévy triplet $(b,\sigma^2,\nu)$ and $\Delta>0$. The  cutoffs $m^\star$ achieving the square bias-variance compromise are:
\begin{align*}
m^\star=  \begin{cases}
\sqrt{\frac{\log n}{\Delta \sigma^2}} &\text{ if } (b,\sigma^2,\nu)\in \mathcal T_{+},\\
\frac{\pi}{2}\Big( \frac{\log n}{M\Delta}\Big)^{\frac{1}{\alpha}}&\text{ if } (b,\sigma^2,\nu)\in\mathcal T_{0,M,\alpha}.
\end{cases}
\end{align*} Moreover, there exists a positive constant $C>0$ such that:
\begin{align*}
\E[\|\hat{f}_{\Delta,m^{\star}}-f_{\Delta}\|^2]\leq C \begin{cases}
 \frac{(\log n)^{\frac12}}{\sigma^2 n \Delta^{\frac12}} &\text{ if } (b,\sigma^2,\nu)\in\mathcal T_{+},\\
 \frac{(\log n)^{\frac1\alpha}}{n(M\Delta)^{\frac{1}{\alpha}}}&\text{ if } (b,\sigma^2,\nu)\in\mathcal T_{0,M,\alpha}.
\end{cases}
\end{align*}
 If $\Delta:=\Delta_{n}\to 0$ as $n\to\infty,$ above results still hold.
\end{theorem}

\subsubsection{General comments}

The dependency in $n$ of the rates is, up to the logarithmic factor, parametric. This is consistent with the fact that on the classes considered the estimated density is $C^{\infty}$. Set aside  the dependency in $\Delta$, the upper bound in Theorem \ref{thm1} is the same as the ones obtained in \cite{butucea2004deconvolution,MR2354572,MR2742504}. The optimal rates of convergence for super-smooth densities (see  \eqref{eq:Picard1}) studied therein suggest that the dependency in $n$ of the rate in Theorem \ref{rateUP} is optimal, including the logarithmic term.  Theorem \ref{thmLB} below shows that the dependency in $\Delta$ is also optimal.

Notice that regimes where $\Delta:=\Delta_{n}$ goes to infinity as $n\to\infty$ can also be considered, provided that $\log n/\Delta_{n}\to\infty$ as $n\to\infty$. The latter conditions ensures that $m^{\star}_{n}$ tends to infinity which is a necessary condition to get consistency.

Over the class $\mathcal T _{+},$ where $\sigma> 0,$ the upper bound in Theorem \ref{rateUP} does not depend on $\nu$. In particular, if $(b,\sigma^2,\nu) \in\mathcal T_{+,M,\alpha}$, the associated rate is ${(\log n)^{1/2}}{ \Delta_n^{-1/2}/n}$, which depends neither on  $M$ nor $\alpha$.
In general, a user has no prior knowledge of $\sigma$ or $\alpha$, which are decisive in the choice of $m^\star$. The problem of finding a data-driven method to select $m^\star$ is addressed in Section \ref{sec:adapt}.  

\subsection{Lower bound result}

Establishing lower bound results for stochastic processes is often technical due to the dual asymptotic regime: the number of observations, $n$, tends to infinity, while the sampling rate, $\Delta$, may tend to 0. Consequently, the statistical model is indexed by both parameters, $n$ and $\Delta$. As the order of the bias term heavily depends on the Lévy triplet, we state a lower bound result separating the cases depending wether the diffusion coefficient  is non-zero or not.

 \begin{theorem}\label{thmLB}
\begin{itemize}
\item[i).] Let $X$ be a Lévy process with Lévy triplet $(b,\sigma^{2},\nu)\in \mathcal T_+$ and denote by $f_\Delta\in \mathcal F_+^\Delta$ the density of $X_\Delta$, $\Delta>0$. Then, there exists a universal constant $C>0$ such that:
$$\inf_{\hat f_{\Delta}}\sup_{f_{\Delta}\in  \mathcal F_+^\Delta}\E[\|\hat f_{\Delta}-f_{\Delta}\|^{2}]\geq \frac{C}{n\sqrt\Delta}.$$
\item[ii).]  Let $\alpha\in(0,2)$, $0<\Delta<  e^{-\frac{4}{2-\alpha}}$ and \begin{align}\label{eq:M}M\leq \frac{2}{\pi}\1_{\alpha=1}+ \frac{\alpha}{(2-\alpha)\cos(\pi/2 \alpha)\Gamma(1-\alpha)}\1_{\alpha\neq 1}.\end{align}
 Let $X$ be a Lévy process whose Lévy triplet is in $\mathcal T _{0,M,\alpha},$ such that the density $f_\Delta$ of $X_{\Delta}$ belongs to $\mathcal F_{0,M,\alpha}^\Delta$. Then, there exist constants $K,c>0$ such that  for any $n$ satisfying $n\log^2(\Delta)\geq K$ it holds:
$$\inf_{\hat f_{\Delta}}\sup_{f_{\Delta}\in \mathcal F_{0,M,\alpha}^\Delta}\E[\|\hat f_{\Delta}-f_{\Delta}\|^{2}]\geq \frac{c}{n\Delta^{\frac{1}{\alpha}}},$$
where $c$ is a constant depending only on $\alpha$.
\end{itemize}
In both cases, the infimum is taken over all possible estimators $\hat f_{\Delta}$ of $f_{\Delta}$.
\end{theorem}
The proof of Theorem \ref{thmLB} i) relies on classical tools for establishing lower bounds based on two hypotheses. Though, we provide a proof in Section \ref{sec:proofs} as the dependence in $\Delta$ could appear unexpected for a Gaussian model. The comparison with the proof of  Theorem \ref{thmLB} ii) illustrates the significant increase in difficulty when transitioning to the pure jump case. In this setting, closed-form expressions for marginal densities of $X_{\Delta}$ are unavailable in general, making lower bounds based on two or more hypotheses particularly challenging. Even when restricting to specific densities within the class $\mathcal F_{0,M,\alpha}^\Delta$, such as symmetric stable densities, sharp bounds for the Hellinger distance, or similar metrics, are, to our knowledge, not available in the literature. Consequently, we 
propose a novel method -generalizable to other situations- that fully exploits the LAN property of symmetric stable densities. 
The key steps of the proof of Theorem \ref{thmLB} ii) are the following.
First in Step 1, we notice that the initial problem is harder than that of estimating densities belonging to $\mathcal S_{\alpha}$, the class of symmetric $\alpha_0$-stable densities with $\alpha_0\in[\alpha,2)$, i.e. 
\begin{align*}\inf_{\hat f_{\Delta}}\sup_{f_{\Delta}\in \mathcal F_{0,M,\alpha}^\Delta}\E[\|\hat f_{\Delta}-f_{\Delta}\|^{2}]\geq \inf_{\hat f_{\Delta}}\sup_{f_{\Delta}\in {\mathcal S}_{\alpha}}\E[\|\hat f_{\Delta}-f_{\Delta}\|^{2}]. \end{align*} 
The next step is to link this problem to that of parametric estimation of $\alpha$ from direct observations of a symmetric $\alpha$-stable process observed with sampling step $\Delta$. Since controlling the $L^2$ distance between densities amounts to controlling the $L^2$ distance between the characteristic functions (thanks to Plancherel's theorem), Step 3 provides a lower bound on the difference in the $L^2$ norm between the characteristic functions of stable processes.
By the end of Step 3, we obtain the following lower bound: \begin{align*}\inf_{\hat f_{\Delta}}\sup_{f_{\Delta}\in {\mathcal S}_{\alpha}}\E[\|\hat f_{\Delta}-f_{\Delta}\|^{2}]\geq\inf_{\hat \alpha\in[0,2]}\sup_{\alpha_0\in[\alpha,\alpha+\frac{2}{\log(1/\Delta)}]}\frac{1}{n\Delta^{1/\alpha}} \E[n \log^2(\Delta)|\hat \alpha-\alpha_0|^2]. \end{align*} Finally, Step 4 shows that the supremum in the right hand side is bounded from below by a constant,  concluding the proof. The two main ingredients in the proof of Step 4 are the LAN property of the family of stable distributions (see Masuda \cite{belomestny2015levy}) and a result on minimax rates for parametric problems under the LAN condition developed in \cite{ibragimov2013statistical}.
Theorem \ref{thmLB}, allows to assert that the rates found in Theorem \ref{rateUP} are nearly minimax (up to a logarithmic factor). This was by no means obvious. While it was clear that the dependency in 
$n$ of the rates could not be improved, nothing could be said about their dependency in $\Delta$ and $\alpha$. For fixed $\Delta$, the problem is marginal, but it becomes fundamental in high-frequency regimes. As far as we know, this is the first minimax optimal result for estimating the density of Lévy processes in such generality.

\subsection{Interpretation of the rate of convergence} 

On the class $\mathcal T_{+,M,\alpha}$ the rate depends on the parameter $\alpha$ through the term $\Delta^{-1}\log n>1$ for $n$ large enough or $\Delta=\Delta_{n}\to 0$. Since $\alpha\mapsto (\Delta^{-1}\log n)^{1/\alpha}$ is non increasing, the fastest rate of convergence is attained for the largest value of $\alpha$ such that $\mathcal T_{+,M,\alpha}$ is satisfied. Introduce the edge parameter 
\begin{align}\label{eq:alpha0}\alpha_{0}:=\sup\left\{\alpha\in(0,2),\  \lim\inf_{\eta\to 0}\frac{\int_{[-\eta,\eta]}x^{2}p(x)dx}{\eta^{2-\alpha}}>0\right\}. \end{align} As $\eta\in(0,1]$, the function $\alpha\mapsto \eta^{2-\alpha}$ is an increasing function of $\alpha$, Condition  \eqref{Ass:p} becomes more stringent as $\alpha$ grows. 
The results in the previous sections show that for a Lévy density $p$ with parameter $\alpha_{0}$, as in \eqref{eq:alpha0}, the rate of convergence of $\hat{f}_\Delta$ lies between $({n\Delta^{1/(\alpha_{0} - \varepsilon)}})^{-1}$ and $({n\Delta^{1/(\alpha_{0} + \varepsilon)}})^{-1}$ (up to constants) for any $\varepsilon > 0$.

Furthermore, some connections between $\alpha_{0}$ and the Blumenthal-Getoor index $\beta$ of $X$ can be established. The following proposition provides a sufficient condition for $\alpha_{0}$ to be the Blumenthal-Getoor index of $X$, and its proof can be found in Section \ref{subsec:proofpropBG}.

\begin{proposition}\label{propBG}
For any Lévy measure $\nu$, it holds that $\alpha_0\leq \beta$.
    If $\nu$ admits a density $p$, $\nu(dx) = p(x) dx,$ such that the limit 
    \begin{equation}\label{cond}\lim_{x \to 0} |x|^{1+\gamma} p(x)\quad  \text{exists for all } \gamma < \beta,
    \end{equation}
    then one gets $\alpha_0 = \beta$.
\end{proposition}
Alternative hypotheses implying  $\alpha_0 = \beta$ exist; for example, assuming that the characteristic exponent of $X$ is regularly varying at infinity as established in Proposition 6.3 in \cite{belomestny2010spectral}. Example \ref{ex:BGnot} below shows that Condition \eqref{cond} is not necessary.

\begin{example}\label{ex:BGnot}
Let $0\le\alpha<\beta< 2$, and consider the Lévy density defined as 
\[p_{0}(x):=\left(\frac{1}{|x|^{\alpha+1}}+\frac{1}{|x|^{\beta+1}}\frac{1+\sin\big({|x|^{-1}}\big)}{2}\right)\mathbf{1}_{x\ne 0}.\] 
Such a function oscillates infinitely often in a neighborhood of 0 between the levels $|x|^{-\alpha-1}$ and $|x|^{-\beta-1}$, and is not regularly varying around 0. Nevertheless, standard computations based on classical integral and series comparisons show that for this Lévy density  $\alpha_{0}=\beta$.
\end{example}

Finally, we show that it is possible to have $\alpha_0 < \beta$  with a pathological example;  the proof is postponed to Section \ref{subsec:proofexample}.
\begin{example}\label{ex:alphaBG}
Let $\eta_k = 2^{-2^k}$, {$k\in\N$}, which defines a partition $I_k = (\eta_{k+1}, \eta_k]$ of $[0, \frac{1}{2}]$, and consider $I_{even} = \bigcup_{k \geq 0} I_{2k}$ and $I_{odd} = \bigcup_{k \geq 0} I_{2k+1}$. For the Lévy density defined as $$p(x) = \frac{1}{x^2}\mathbf{1}_{I_{odd}}(x)+ \frac{1}{x^\frac{3}{2}}\mathbf{1}_{I_{even}}(x),$$ it holds that $\alpha_0 = \frac{1}{2}$ and $\beta= 1$. Clearly, exponents may be adjusted to achieve any value of $\alpha_0$ and $\beta$.
\end{example}

\subsection{Adaptive estimator\label{sec:adapt}}

We propose an adaptive estimator that  attains the bound of Theorem \ref{thm1}. This procedure is inspired by the recent thresholded procedure introduced in \cite{ammous2024adaptive}. If the procedure induces a logarithmic loss, it has the advantages of that its proof  relies on simple arguments, is simple to implement and performs well numerically. A penalized approach, as developed in \cite{comte2010nonparametric}, would also lead to an adaptive estimator.

Let $\kappa>0$ be a positive constant and set $\kappa_{n}:=(1+\kappa\sqrt{\log n})$. Consider the estimator of $f_\Delta$ defined by
\begin{align}\label{eq:estfbis}\tilde{f}_\Delta(x)  = \frac{1}{2\pi} \int_{[-n,n]}e^{- i ux} \tilde{\phi}_{X_{\Delta}}(u) { d}u,\quad x\in{\mathbb R} {,} 
\end{align}
where 
\begin{align}
\tilde{\phi}_{X_{\Delta}}(u)=\hat{\phi}_{X_{\Delta}}(u) \mathbf{1}_{\{|\hat{\phi}_{X_{\Delta}}(u)|\geq {\kappa}_n n^{-1/2}\}},\quad u\in{\mathbb R}.
\end{align} 

\begin{theorem}\label{thm:adapt} Under the assumptions of Theorem \ref{thm1}, the adaptive estimator satisfies the following inequality, for any $\kappa>0$,\begin{align*}
{\mathbb E}[  \| \tilde{f}_\Delta- f_{\Delta}\|^2 ]&  \leq \inf_{m\in(0,n]}\left\{9 \|f_{\Delta,m}-f_{\Delta}\|^{2} +\frac m{\pi n} (5+(1+(\kappa+2)\sqrt{\log n})^{2})\right\} +64n^{1-\frac{\kappa^{2}}{4}}.\end{align*}
\end{theorem}

Theorem \ref{thm:adapt} is an oracle inequality and ensures, in particular, that the adaptive estimator $\tilde f_{\Delta}$ attains the optimal rate of convergence over the classes introduced in Definition \ref{def:T} for any $\kappa\ge2\sqrt2$.
Proof of Theorem \ref{thm:adapt} follows the lines of the proof of Theorem 1 in \cite{ammous2024adaptive}. However, since the procedure is new and the context slightly different, with the law of observations depending on the sampling rate $\Delta$, the proof is provided and tailored to our specific setting.

\subsection{Case of vanishing Gaussian component}

The fact that the rate of convergence of $\hat{f}_{\Delta,m}$ is unaffected by the presence or absence of the Lévy measure when $\sigma\neq 0$ is intuitive when jumps are rare. However, in a scenario where the jump activity  is exceedingly high while the volatility $\sigma$ is relatively low, though  positive, we anticipate that the presence of jumps might influence the final convergence rate. We study this framework on the class $\mathcal T_{+,M,\alpha}$, see Definition \ref{def:T}. Let $\sigma=\sigma_n\to0$ as $n\to\infty$, the rate of convergence of the estimator is closely related to the rate at which $\sigma_n$ decays. In this context, we identify two distinct behaviors leading to different convergence rates. In the following, while $\Delta>0$ remains fixed, its role is  explicitly included in the asymptotic expressions to emphasize its effect on the convergence rates.
\begin{theorem}\label{thm:rate_sigma}
Let $(X_n)_{n>0}$ be a sequence of Lévy processes with triplets $(b,\sigma^2_n,\nu)\in \mathcal{T}_{+,M,\alpha}$, for some $M>0$, and $\alpha \in (0,2)$. Let $\Delta>0$ be fixed. There exist $C,\tilde C>0$ positive constants only depending on $M$ and $\alpha$ such that:
\begin{enumerate}
\item If $\sigma_n (\log n)^{1/\alpha-1/2} \underset{n\rightarrow \infty}{\rightarrow} \infty$, selecting $m_{n}^\star=O\left(\sqrt{{\log n}/({\Delta \sigma_n^2})}\right)$ leads to the rate
\begin{align*}
\E[\|\hat{f}_{\Delta,m_{n}^\star}-f_{\Delta}\|^2]\leq  C \frac{\sqrt{\log n}}{\sigma_n n \Delta^{1/2}}.
\end{align*}
\item If $\sigma_n ( \log n)^{1/\alpha-1/2}{=}O(1)$ as $n\to \infty$, selecting $m_{n}^\star=O\left({(\log n)^{1/\alpha}}{\Delta^{-1/\alpha}}\right)$ leads to the rate \begin{align*}
\E[\|\hat{f}_{\Delta,m^\star}-f_{\Delta}\|^2]\leq    \tilde C \frac{(\log n)^{1/\alpha}}{n\Delta^{1/\alpha}}.
\end{align*}
\end{enumerate}
\end{theorem}
The convergence rates in Theorem \ref{thm:rate_sigma} are determined by  the interplay between the Brownian component and the contribution from jumps. When $\sigma_n$ is asymptotically slower than $(\log n)^{1/2-1/\alpha}$,  the convergence rate  is similar to the case where $\sigma_n$ is fixed. This result indicates that when the contribution of jumps is negligible, the rate remains governed by the Brownian component. 
In the second case, the decay of $\sigma_n$ is either of the order of $(\log n)^{1/2 - 1/\alpha}$, with convergence rates matching those observed in the pure Brownian motion setting or in the pure jump context.
Alternatively, if $\sigma_n$ decays faster than $(\log n)^{1/2 - 1/\alpha}$, the convergence rate coincides with that of the pure jump model.
As the influence of the Brownian component diminishes, the convergence rate becomes increasingly dominated by the jumps.

\section{Numerical illustration}\label{sec:numerics}

We illustrate the numerical performance of the adaptive estimator $\tilde{f}_\Delta$ defined in \eqref{eq:estfbis} for three pure jump stable processes, whose increments are easy to implement. In particular, we focus on stable Lévy processes $X$ with triplet $(0,0,\nu)$ and Lévy density 
\[
p(x) = \frac{P}{x^{1+\alpha}}\ind_{x>0} + \frac{Q}{|x|^{1+\alpha}} \ind_{x<0}
\]
belonging to $\mathcal{T}_{ \frac{P+Q}{2-\alpha},\alpha}$, where $\alpha \in (0,2)$. The distribution of $X_{\Delta}$ is then a stable distribution with parameters
{\small
\[
\left( \alpha, \Delta^{1/\alpha} \left( \frac{\Gamma(1-\alpha)}{\alpha} (P+Q)\cos(\pi \alpha/2) \right)^{1/\alpha}, \frac{P-Q}{P+Q}, \frac{Q-P}{2-\alpha} \Delta \right).
\]
}
In general, the associated density $f_{\Delta}$ is not explicit, but efficient algorithms to compute it numerically are available (see \cite{nolan2020univariate}). We consider three cases: a finite variation process with $(P,Q,\alpha) = (2,1,0.7)$, an infinite variation process with $(P,Q,\alpha) = (2,1,1.7)$, and a symmetric Cauchy process with $(P,Q,\alpha) = (\pi^{-1}, \pi^{-1}, 1)$. In the latter case, the density $f_\Delta$ is explicit and is given by $
f_\Delta(x) = \frac{\Delta}{\pi(x^2 + \Delta^2)},  \text{for } x \in \mathbb{R}.$
The performance of the estimator is evaluated by comparing its risks for different values of $n \in \{500, 1000, 5000\}$ and $\Delta \in \{0.1, 1\}$. To facilitate comparison between the different examples, for which $\|f_\Delta\|^2$ may vary, we compute an empirical version of the relative $L^2$ error, defined as $\frac{\E[\| \tilde{f}_\Delta - f_\Delta \|^2]}{\| f_\Delta \|^2}$, obtained via Monte Carlo simulation over 100 independent trajectories. For computational cost reasons, the domain where $\phi_\Delta$ is estimated is fixed to $[-100,100]$ for $\Delta = 0.1$ and to $[-10,10]$ for $\Delta = 1$. Indeed, when $\Delta$ is small, as the theoretical characteristic function $\phi_\Delta$ approaches 1 over a large interval, a larger integration domain must be considered.

As with most adaptive methods, the estimator $\tilde{f}_{\Delta}$ depends on the choice of a parameter $\kappa$. We apply the calibration strategy proposed in \cite{ammous2024adaptive}, Section 2.3, which relies on the Euler characteristic of the unthresholded areas 
\[
A_{n}(\kappa) := \{u \in [-n, n] :  |\hat{\phi}_{X_\Delta}(u)| \geq (1 + \kappa \sqrt{\log n}) n^{-1/2} \}.
\]
In dimension 1, the Euler characteristic, denoted by $\chi$, counts the number of connected components. Considering the mapping $\kappa \mapsto \chi(A_{n}(\kappa))$, the following global behavior is expected. If $\kappa$ is small, $\chi(A_{n}(\kappa))$ should be large, as even in regions where the true characteristic function is close to 0, fluctuations of the estimated characteristic function make it exceed the threshold. Conversely, when $\kappa$ is large, these fluctuations are no longer visible, and $\chi(A_{n}(\kappa))$ should stabilize.
This motivates the following empirical choice for $\kappa$. Given a uniform grid of possible values for $\kappa$: $\{k\delta, k \in N\}$, for some $N \in \mathbb{N}$ and $\delta > 0$, we select the first value for which the Euler characteristic stabilizes, namely
\[
\hat{\kappa_n} = \inf \{k\delta \mid k \in \{1, \ldots, N\}, \chi(A_{n}(k\delta)) = \chi(A_{n}((k-1)\delta)) \}.
\]
However, numerically this choice is very sensitive to the selection of $\delta$. To reduce this sensitivity, we impose that $\chi(A_{n}(k\delta))$ remains stable for three consecutive values:
\[
\tilde{\kappa}_n = \inf \{ k\delta \mid k \in \{2,\ldots,N\}, \chi(A_{n}(k\delta)) = \chi(A_{n}((k-1)\delta)) = \chi(A_{n}((k-2)\delta))  \}.
\]

\begin{table}[t]
\begin{center}
\begin{tabular}{|c|c|c|c|c|c|}
   \hline 
$\Delta$ & $n$ & 500 & 1000 & 5000 & 10000\\
   \hline 
    \multirow{4}{*}{0.1} &      \multirow{2}{*}{$\tfrac{\nrm{\tilde{f}_{\Delta} - f_\Delta}^2}{\nrm{f_\Delta}^2}$} & $5.62 \times 10^{-1}$ &$3.03 \times 10^{-2}$ &$8.89 \times 10^{-3}$ & $5.21 \times 10^{-3}$\\
     & & \small$(2.42\times 10^{-2})$ &  \small$(3.49 \times 10^{-2})$& \small$(4.65 \times 10^{-3})$ & \small$(2.93\times 10^{-3})$\\
     \cline{2-6}
&    {$\overline{\tilde{\kappa}_n}$} & 1.02 \small(0.53)&0.92 \small(0.56)&0.92 \small(0.36)&1.01 \small(0.39)\\
     \hline 
    \hline 
    \multirow{4}{*}{1} &      \multirow{2}{*}{$\tfrac{\nrm{\tilde{f}_{\Delta} - f_\Delta}^2}{\nrm{f_\Delta}^2}$} & $4.70 \times 10^{-2}$ &  $2.59 \times 10^{-2}$&$ 6.95\times 10^{-3}$ & $ 5.10\times 10^{-3}$ \\
     & & \small$(2.11 \times 10^{-2})$&\small$(1.48 \times 10^{-2})$ &\small$(2.79 \times 10^{-3})$ &\small$(2.45 \times 10^{-3})$\\
     \cline{2-6}
&    {$\overline{\tilde{\kappa}_n}$} & 0.86 \small(0.21) &0.84  \small(0.18) &0.79 \small(0.19) &1.01  \small(0.33)\\
   \hline 
\end{tabular}
\end{center}
\caption{ Mean and standard deviation (in parentheses) of the relative $L^2$ risk and selected parameter $\kappa$ for $100$ estimators for the $0.7-$stable process.
 }
\label{table:stable_FV}

\end{table}

\begin{table}[t]
\begin{center}
\begin{tabular}{|c|c|c|c|c|c|}
   \hline 
$\Delta$ & $n$ & 500 & 1000 & 5000 & 10000\\
   \hline 
    \multirow{4}{*}{0.1} &      \multirow{2}{*}{$\tfrac{\nrm{\tilde{f}_{\Delta} - f_\Delta}^2}{\nrm{f_\Delta}^2}$} & $2.50 \times 10^{-2}$ &$1.36 \times 10^{-2}$ &$3.44 \times 10^{-3}$ & $1.92 \times 10^{-3}$\\
     & & \small$(1.44\times 10^{-2})$ &  \small$(7.36 \times 10^{-3})$& \small$(1.76 \times 10^{-3})$ & \small$(1.01\times 10^{-3})$\\
     \cline{2-6}
&    {$\overline{\tilde{\kappa}_n}$} & 0.69 \small(0.41)&0.78 \small(0.38)&0.78 \small(0.37)&0.80 \small(0.40)\\
     \hline 
    \hline 
    \multirow{4}{*}{1} &      \multirow{2}{*}{$\tfrac{\nrm{\tilde{f}_{\Delta} - f_\Delta}^2}{\nrm{f_\Delta}^2}$} & $2.10 \times 10^{-2}$ &  $1.25 \times 10^{-3}$&$ 3.03\times 10^{-3}$ & $ 1.75\times 10^{-3}$ \\
     & & \small$(1.10 \times 10^{-2})$&\small$(7.34 \times 10^{-3})$ &\small$(1.53 \times 10^{-3})$ &\small$(9.24 \times 10^{-4})$\\
     \cline{2-6}
&    {$\overline{\tilde{\kappa}_n}$} & 0.68 \small(0.39) &0.76  \small(0.36) &0.78 \small(0.36) &0.77  \small(0.34)\\
   \hline 
\end{tabular}
\end{center}
\caption{ Mean and standard deviation (in parentheses) of the relative $L^2$ risk and selected parameter $\kappa$ for $100$ estimators for the $1$-stable process.
 }
\label{table:cauchy}

\end{table}

\begin{table}[t]
\begin{center}
\begin{tabular}{|c|c|c|c|c|c|}
   \hline 
$\Delta$ & $n$ & 500 & 1000 & 5000 & 10000\\
   \hline 
    \multirow{4}{*}{0.1} &      \multirow{2}{*}{$\tfrac{\nrm{\tilde{f}_{\Delta} - f_\Delta}^2}{\nrm{f_\Delta}^2}$} & $3.16 \times 10^{-2}$ &$1.70 \times 10^{-2}$ &$2.37 \times 10^{-3}$ & $8.48 \times 10^{-4}$\\
     & & \small$(3.23\times 10^{-2})$ &  \small$(1.88 \times 10^{-2})$& \small$(2.37 \times 10^{-3})$ & \small$(1.01\times 10^{-3})$\\
     \cline{2-6}
&    {$\overline{\tilde{\kappa}_n}$} & 0.41 \small(0.14)&0.39 \small(0.15)&0.41 \small(0.12)&0.44 \small(0.08)\\
     \hline 
    \hline 
    \multirow{4}{*}{1} &      \multirow{2}{*}{$\tfrac{\nrm{\tilde{f}_{\Delta} - f_\Delta}^2}{\nrm{f_\Delta}^2}$} & $3.21 \times 10^{-2}$ &  $1.50 \times 10^{-2}$&$ 1.85\times 10^{-3}$ & $ 8.45\times 10^{-4}$ \\
     & & \small$(2.60 \times 10^{-2})$&\small$(1.21 \times 10^{-2})$ &\small$(1.82 \times 10^{-3})$ &\small$(6.99 \times 10^{-4})$\\
     \cline{2-6}
&    {$\overline{\tilde{\kappa}_n}$} & 0.25 \small(0.13) &0.25  \small(0.13) &0.32 \small(0.12) &0.32  \small(0.12)\\
   \hline 
\end{tabular}
\end{center}
\caption{ Mean and standard deviation (in parentheses) of the relative $L^2$ risk and the selected parameter $\kappa$ for 100 estimators for the $1.7$-stable process.}
\label{table:stable_IV}
\end{table}

As expected, we observe in Tables \ref{table:stable_FV}, \ref{table:cauchy}, and \ref{table:stable_IV} improvements in the risks as $n$ increases and for larger values of $\Delta$ and $\alpha$. This comparison is possible because we consider the relative risks rather than the classical MISE. These results are consistent with the theoretical rates of convergence.
The selected values for $\kappa$ seem to vary with $\Delta$ and $\alpha$, but not with $n$. This is in line with the fact that the domain where $\phi_\Delta$ takes large values depends on $\alpha$ and $\Delta$, but not on $n$.

\section{Proofs}\label{sec:proofs}

\subsection{Proof of Theorem \ref{rateUP}}
To derive the rates of convergence we perform a square bias and variance compromise. Define $V(m)=m/(\pi n)$. Controls over the square bias term are obtained via \eqref{biais}, we easily derive:
\begin{itemize}
\item If $X$ is a Lévy process with Lévy triplet $(b,\sigma^2,\nu)\in\mathcal T_{+}$, then 
\begin{align}\label{eq:BiasB}\|f_{\Delta,m}-f_{\Delta}\|^{2}\leq \frac{1}{\pi}\int_m^{\infty} e^{-\Delta \sigma^2 u^2}du=:b^{2}_{1}(m),\quad \forall \Delta>0,\ m\geq 0.\end{align}
\item Let $M>0$ and $\alpha\in(0,2)$ be given. If the Lévy triplet $(b,\sigma^2,\nu)\in\mathcal T_{0,M,\alpha}$, then  Lemma \ref{lemmapicard} in the Appendix allows to derive 
\begin{align}\label{eq:BiasJ}\|f_{\Delta,m}-f_{\Delta}\|^{2}\leq \frac{\Gamma\Big(\frac{1}{\alpha}, \big(\frac{2m}{\pi}\big)^\alpha M\Delta\Big)}{2\alpha (M\Delta)^{\frac{1}{\alpha}}}=:b^{2}_{2}(m),\quad \forall \Delta>0,\ m\geq \frac{\pi}{2}.\end{align}
\end{itemize}
If $(b,\sigma^2,\nu)\in\mathcal T_+$, we compute the minimizer $m^\star$ of $b^{2}_{1}(m)+V(m)$. By computing the derivative, we find that $
m^\star = \sqrt{{\log(n)}{/(\Delta \sigma^2)}}.
$
Using the equivalent $\int_{x}^\infty e^{-t^2}dt \underset{x\rightarrow \infty}{\sim} \frac{e^{-x^2}}{2x}$, for any fixed $\Delta,\sigma>0$, the square bias term is asymptotically equivalent to
\begin{align}\label{eq:bias_equiv}
b^2_{1}(m^\star) \underset{n\rightarrow \infty}{\sim} \frac{1}{\pi \sqrt{\Delta} \sigma n\sqrt{\log(n)}}.
\end{align}
The associated variance term is $V(m^\star)= \frac{\sqrt{\log n}}{\sqrt{\Delta} \sigma n},$ the  bound is therefore variance-dominated. If $\Delta=\Delta_n\to 0$ as $n\rightarrow \infty$, \eqref{eq:bias_equiv} remains valid, since $\sqrt{\Delta_n} m^\star \underset{ n\rightarrow \infty}{\rightarrow} \infty,$ and the global error is still asymptotically variance-dominated since
$
b^2(m^\star) \underset{n\rightarrow \infty}{\sim} V(m^\star) (\pi \log n)^{-1}.
$

The case $(b,\sigma^2,\nu)\in\mathcal T_{0, M,\alpha}$ is treated in the same way. Differentiating $m\mapsto b_{2}^{2}(m)+V(m)$, allows to find that  $m^\star$ satisfying
$e^{-M\Delta({2m^\star}/{\pi})^\alpha}={n^{-1}}$ is the announced value. The risk is again asymptotically variance-dominated, and  is of order of ${m^\star}/{n}$.

\subsection{Proof of Proposition \ref{propBG}} \label{subsec:proofpropBG}
Recall the definitions
\begin{align*}
\alpha_0 &= \sup \left\{
\gamma \in [0, 2) : \varphi(\gamma) := \liminf\limits_{\eta \to 0} \eta^{\gamma-2} \int_{-\eta}^{\eta} x^2 \nu(dx) > 0
\right \},\\
\beta &= \inf \left \{
    \gamma \in [0, 2] : \psi(\gamma):= \int_{-1}^{1} |x|^\gamma \nu(dx) < \infty
\right \}.
\end{align*}
First, remark that the sets defining $\alpha_0$ and $\beta$ are intervals, i.e. any $\gamma < \alpha_0$ belongs to the set defining $\alpha_0$ and any $\gamma > \beta$ belongs to the set defining $\beta$.

$\alpha_0 \leq \beta$: This part of the proof does not make use of any assumption. Suppose by contradiction that $\beta < \alpha_0$ and choose $\beta < \gamma_1 < \gamma_2 < \alpha_0$, so that $\psi(\gamma_1) < \infty$ and $\varphi(\gamma_2) > 0$. But then:
$$
\varphi(\gamma_2) = \liminf_{\eta \to 0}\eta^{\gamma_2 -2} \int_{-\eta}^\eta x^2 \nu(dx) \leq \liminf_{\eta \to 0} \eta^{\gamma_2 - \gamma_1} \int_{-\eta}^\eta |x|^{\gamma_1} \nu(dx) = 0.
$$

$\alpha_0 \geq \beta$: Apply l'Hôpital's rule to the limit:
\begin{equation}\label{eqn:lim}
\lim_{\eta \to 0} \frac{\int_{-\eta}^\eta x^2 p(x) dx}{\eta^{2-\gamma}}.
\end{equation}
Since $\gamma < 2$, and $\int x^2 p(x) dx < \infty$, both the numerator and denominator tend to 0 as $\eta \to 0$. Additionally, they are both differentiable, and the derivative of the denominator does not vanish. Finally, by hypothesis, $ \lim_{\eta \to 0}  \eta^{1+\gamma} p(\eta)$ exists. Therefore, the limit in \eqref{eqn:lim} exists and equals $\varphi(\gamma) = \lim_{\eta \to 0} \frac{2}{2-\gamma} \eta^{1+\gamma} p(\eta)$.
Now, assume for contradiction that $\alpha_0 < \beta$. Then, there exist $\gamma_1$ and $\gamma_2$ such that $\alpha_0 < \gamma_1 < \gamma_2 < \beta$, meaning $\varphi(\gamma_1) = 0$, and $\psi(\gamma_2) = \infty$. Since $\varphi(\gamma_1) = 0$, for any $\varepsilon > 0$, there exists a $\delta > 0$ such that $|x|^{\gamma_1+1} p(x) < \varepsilon$ for all $x$ with $|x| < \delta$. On the other hand, since $p(x)$ is a Lévy density, the integral in $\psi$ can be restricted to an arbitrarily small neighborhood of 0, leading to:
$$
\infty = \int_{-\eta}^{\eta} |x|^{\gamma_2} p(x) dx = \int_{-\eta}^{\eta} \frac{|x|^{\gamma_2 - \gamma_1}}{|x|} |x|^{\gamma_1+1} p(x) dx < \varepsilon \int_{-\eta}^\eta |x|^{\gamma_2-\gamma_1-1} dx,
$$
which is a contradiction, since $\gamma_2 - \gamma_1 - 1 > -1$.

\subsection{Proof of Example \ref{ex:alphaBG}} \label{subsec:proofexample}
  Consider the following two Lévy densities:
$$
p_{even}(x) = \begin{cases}
    \frac{1}{x^2} & \text{if } x \in I_{even},\\
    \frac{1}{x^\frac{3}{2}} & \text{if } x \in I_{odd},
\end{cases} \quad
p_{odd}(x) = \begin{cases}
    \frac{1}{x^\frac{3}{2}} & \text{if } x \in I_{even},\\
    \frac{1}{x^2} & \text{if } x \in I_{odd}.
\end{cases}
$$
Then, since $p_{even} + p_{odd} > \frac{1}{x^2}$, we obtain:
$$
\int_{|x|<1} |x|^\gamma \frac{1}{x^2} \, dx \leq \int_{|x|<1} |x|^\gamma p_{even}(x) \, dx + \int_{|x|<1} |x|^\gamma p_{odd}(x) \, dx.
$$
This shows that, if $\gamma < 1$ (1 is the Blumenthal-Getoor index of $\frac{1}{x^2}$), then at least one (in fact, both) of the two integrals on the right must diverge, meaning that at least one between $p_{even}$ and $p_{odd}$ has $\beta \ge  1$. We will demonstrate, however, that $\alpha_0 \le \frac{1}{2}$. From now on, we focus on $p = p_{odd}$, but the proof for $p_{even}$ is analogous (using $\eta_{2k+1}$ instead of $\eta_{2k}$).
We claim that, for every $\gamma > \frac{1}{2}$, the following integral:
$$
a_k = \eta_{2k}^{\gamma-2} \int_0^{\eta_{2k}} x^2 p(x) \, dx
$$
tends to 0. Indeed, this can be written:
\begin{align*}
a_k &= \eta_{2k}^{\gamma - 2} \bigg ( \sum_{i \geq k} \int_{\eta_{2i+2}}^{\eta_{2i+1}} dx + \sum_{i \geq k} \int_{\eta_{2i+1}}^{\eta_{2i}} \sqrt{x} dx \bigg) \\
&=
 \eta_{2k}^{\gamma-2} \bigg (\sum_{i=k}^\infty \big(\eta_{2i+1} - \eta_{2i+2}\big) + \frac{2}{3}  \sum_{i=k}^\infty \big( \eta_{2i}^\frac{3}{2} - \eta_{2i+1}^\frac{3}{2}\big)\bigg)\\
& \leq \eta_{2k}^{\gamma-2} \bigg( \sum_{i=k}^\infty \eta_{2i+1} + \sum_{i=k}^\infty \eta_{2i}^\frac{3}{2}\bigg) =
 \eta_{2k}^\gamma 2^{2^{2k+1}}\sum_{i=k}^\infty 2^{-2^{2i+1}} + \eta_{2k}^\delta 2^{\frac{3}{2}2^{2k}} \sum_{i=k}^\infty \left(2^{-2^{2i}}\right)^{\frac32}
\end{align*}
where we have set $\gamma = \delta + \frac{1}{2}$. This leads to:
$$
a_k \leq  \eta_{2k}^\gamma \sum_{i=k}^\infty 2^{2^{2k+1} - 2^{2i+1}} + \eta_{2k}^\delta \sum_{i=k}^\infty \Big(2^{2^{2k} - 2^{2i}} \Big)^{\frac{3}{2}}.
$$
Let us introduce $h = i - k$, so that the exponents in the sum become $2^{2k+1} - 2^{2k+2h+1}$ and $2^{2k} - 2^{2k + 2h}$. Note that, for any positive integers $n, m > 0$, we have $2^{n} - 2^{n+m} \leq -2^{n+m-1}$. Applying this with $n = 2k+1$ (resp. $n = 2k$ in the second sum) and $m = 2h$ for $h > 0$, we get:
$$
a_k \leq \eta_{2k}^\gamma \left(1 + \sum_{h=1}^\infty 2^{-2^{2h+2k}}\right) + \eta_{2k}^\delta \left(1 + \sum_{h=1}^\infty 2^{-\frac{3}{2} 2^{2h+2k-1}}\right).
$$
The terms in the parentheses are finite number; therefore, for any $\gamma > \frac{1}{2}$, so that $\delta > 0$, $a_k$ tends to 0, as claimed. This implies that $\alpha_{0}\le \frac12$. 

\newpage

\subsection{Proof of Theorem \ref{thmLB} }
\subsubsection{Proof of Theorem \ref{thmLB} i)}
Proof of Theorem \ref{thmLB} i) uses standard techniques to derive lower bounds based on two hypotheses, as illustrated in Chapter 2 of \cite{tsy}. For density estimation problems from $n$ i.i.d. observations, the minimax risk 
$$\mathcal R_{n}:=\inf_{\hat f_{\Delta}}\sup_{f_{\Delta}\in  \mathcal F_+^\Delta}\E[\|\hat f_\Delta-f_\Delta\|^2]$$ 
is lower bounded by the risk associated with the choice of two specific parameters in the parameter space $\mathcal F_+^\Delta$. In our case, we choose two (sequences of) Lévy processes $X_1$ and $X_2$ with Lévy triplets given by $(0,1,0)$ and $(0,1+\frac{1}{\sqrt n},0)$, respectively. For any $\Delta>0$, we denote by $f_{1,\Delta}$ (resp. $P_{f_{1,\Delta}}$) and $f_{2,\Delta}$ (resp. $P_{f_{2,\Delta}}$) the densities of $X_1$ and $X_2$ at time $\Delta$ (resp. the law of $X_1$ and $X_2$ at time $\Delta$). In particular, both $f_{1,\Delta}$ and $f_{2,\Delta}$ belong to $\mathcal F_+^\Delta$. General arguments lead to the following lower bound:
$$\mathcal R_{n}\geq \frac{1}{4} \|f_{1,\Delta}- f_{2,\Delta}\|^2\Big\{1-\frac{1}{2}\|P_{f_{1,\Delta}}^{\otimes n}-P_{f_{2,\Delta}}^{\otimes n}\|_1\Big\}.$$
By using the fact that the $L^1$ distance is upper bounded by the Hellinger distance $H$, and that it is easy to control such a distance for product measures, this leads to
\begin{align*}
\mathcal R_{n}\geq \frac{1}{4} \big(1-\sqrt n H(P_{f_{1,\Delta}}, P_{f_{2,\Delta}})\big)\|f_{1,\Delta}- f_{2,\Delta}\|^2.
\end{align*}
Now, using the fact that $f_{1,\Delta}$ is the density of a Gaussian distribution $\mathcal N(0,\Delta)$, and $f_{2,\Delta}$ that of a Gaussian distribution $\mathcal N(0,\Delta(1+1/\sqrt n))$, and that $H^2(\mathcal N(0,\sigma_1^2), \mathcal N(0,\sigma_2^2))=1-\sqrt{\frac{2\sigma_1\sigma_2}{\sigma_1^2+\sigma_2^2}}$, we derive
\begin{align*}
\mathcal R_{n} \geq \frac{1}{4}\left(1-\sqrt n \sqrt{1-\sqrt{\frac{2\sqrt{1+n^{-1/2}}}{2+n^{-1/2}}}}  \right)\|f_{1,\Delta}- f_{2,\Delta}\|^2.
\end{align*}
First, for all $y>0$ and $\eps\in(0,1],$ it holds that $e^{y}-e^{y(1-\eps)}\ge \eps y$. From this, we can deduce (setting $y=\Delta u^{2}/2$ and $\eps=1/\sqrt{n}$) and using  Plancherel's equality, that  for any $n\geq 1$, $u\ge0$, we have 

\begin{align*}
\|f_{1,\Delta}- f_{2,\Delta}\|^2 &=\frac1\pi\int_{0}^{\infty}\big(e^{-\frac{\Delta u^2}{2}}-e^{-\frac{\Delta u^2 (1+n^{-1/2})}{2}}\big)^2du \\
&\geq \frac1\pi\int_{0}^{\infty}\frac{\Delta^2 u^4 e^{-2u^2\Delta}}{4n}du= \frac{3}{2^{7}\sqrt {2\pi} n\sqrt\Delta }.
\end{align*}

Let $h(n):=n \Big(1-\sqrt{\frac{2\sqrt{1+n^{-1/2}}}{2+n^{-1/2}}}\Big)$, $\forall n\in\N^*$. Since $h$ is an increasing function and 
$\lim_{n\to\infty} h(n)=\frac{1}{16}$, we conclude that $\sqrt n \sqrt{1-\sqrt{\frac{2\sqrt{1+n^{-1/2}}}{2+n^{-1/2}}}}\leq \frac 1 4$, and therefore $\mathcal R_{n}\geq \frac{C}{n\sqrt \Delta}$ with $C=\frac{9}{2048\sqrt{2\pi}}.$ This completes the proof of Theorem \ref{thmLB} i).

\subsubsection{Proof of Theorem \ref{thmLB} ii)}

The lower bound is established by showing that the announced rate is attained on the class of symmetric $\alpha_0$-stable processes, for $\alpha_0\in[\alpha,2)$. The proof is divided in four main steps.

\emph{Step 1:} First observe that 
$$ \inf_{\hat f_{\Delta}}\sup_{f_{\Delta} \in {\mathcal F}_{0,M,\alpha}^{\Delta} }\E\|\hat f_{\Delta}-f_{\Delta}\|^{2} \ge \inf_{\hat f_{\Delta}}\sup_{f_{\Delta}\in {\mathcal S}_{\alpha}}\E\|\hat f_{\Delta}-f_{\Delta}\|^{2},$$
where $\mathcal S_{\alpha}$ denotes the class of symmetric $\alpha_0$-stable densities with characteristic function $e^{-\Delta |u|^{\alpha_0}}$, $\alpha_0\in[\alpha,2)$.
In particular, if $f_\Delta$ is the density of $X_\Delta$ with $\E[e^{iuX_\Delta}]=e^{-\Delta|u|^{\alpha_{0}}}$, then the L\'evy density of $X$ is given by $p_{\alpha_0}(x) = \frac{P_{\alpha_0}}{|x|^{1+\alpha_0}}\ind_{x\neq 0}$, where $$P_{\alpha_0}:=\left(2\cos\left(\frac{\pi\alpha_{0}}{2} \right) \alpha_0^{-1} \Gamma(1-\alpha_0)\right)^{-1}\1_{\alpha_0\neq 1}+\frac1\pi \1_{\alpha_0=1},$$ see e.g. \cite{belomestny2015levy} p.215. Therefore, for all $\eta \in (0,1)$, one has 
$ \int_{[-\eta,\eta]} x^2 p_{\alpha_0}(x) dx =\frac{ 2 P_{\alpha_0}}{2-\alpha_0} \eta^{2-\alpha_0}$ where  $\alpha_0\mapsto \frac{ 2 P_{\alpha_0}}{2-\alpha_0}$ is increasing in $[\alpha,2)$. It follows that, for any $\eta\in(0,1)$, $ \int_{[-\eta,\eta]} x^2 p_{\alpha_0}(x) dx\geq \frac{2P_\alpha}{2-\alpha}\eta^{2-\alpha}\geq M \eta^{2-\alpha}$, with $M$ as in  \eqref{eq:M}, and $f_\Delta$ satisfies Assumption \eqref{Ass:p}. Thus, any $f_\Delta\in \mathcal S_{\alpha}$ is an element of ${\mathcal F}_{0,M,\alpha}^{\Delta}$ if $M\leq  \frac{2P_\alpha}{2-\alpha}.$

\emph{Step 2:} Denote by
\begin{align*}
\mathcal R_{\alpha} & := \bigg\{\phi_\Delta: \mathbb{R} \to \mathbb{R} \ : \ \exists \beta \in [\alpha,2) \ \text{such that } \phi_\Delta(t) = e^{-\Delta |t|^\beta}, \forall t \in \mathbb{R} \bigg\}, \\
\tilde{\mathcal R}_{\alpha} & := \bigg\{\phi_\Delta : \exists \beta \in \Big[\alpha, \alpha-\frac{2}{\log(\Delta)}\Big] \subset (0,2) \ \text{such that } \phi_\Delta(t) = e^{-\Delta |t|^\beta}, \forall t \in \mathbb{R} \bigg\},
\end{align*}
and observe that $\tilde {\mathcal R}_{\alpha} \subseteq {\mathcal R}_{\alpha}$, given that $\Delta < e^{-\frac{2}{2-\alpha}}$. Thus, by applying the Plancherel identity, we can express, for $\phi_{\alpha_0,\Delta}(t) := e^{-\Delta |t|^{\alpha_0}}, \ t \in \mathbb{R}$, the following:
\begin{align*}
\inf_{\hat f_{\Delta}} \sup_{f_{\Delta} \in \mathcal S_{\alpha}} \mathbb{E} \|\hat f_{\Delta} - f_{\Delta}\|^{2} 
& = \frac{1}{2\pi} \inf_{\hat \phi_{\Delta}} \sup_{\phi_{\Delta} \in \mathcal{R}_{\alpha}} \mathbb{E} \|\hat \phi_{\Delta} - \phi_{\Delta}\|^{2} \\
& \geq \frac{1}{2\pi} \inf_{\hat \phi_{\Delta}} \sup_{\phi_{\alpha_0,\Delta} \in \tilde{\mathcal{R}}_{\alpha}} \mathbb{E} \|\hat \phi_{\Delta} - \phi_{\alpha_0,\Delta}\|^{2}.
\end{align*}

Next, we demonstrate that the estimation of $\phi_{\alpha_0,\Delta}$ can be reduced to focusing on the parameter $\alpha_0$. For a given estimator $\hat \phi_{\Delta}$, define the quantity
$$
\hat\alpha(\hat\phi_{\Delta}) \in \arg \min_{\alpha_0 \in [\alpha, \alpha + \frac{2}{\log(1/\Delta)}]} \mathbb{E} \|\hat \phi_{\Delta} - \phi_{\alpha_0,\Delta}\|^{2}.
$$
It follows that $\phi_{\hat\alpha(\hat\phi_{\Delta}),\Delta}$ is the characteristic function of the symmetric stable process in $\tilde{\mathcal R}_{\alpha}$ that is closest to the estimator $\hat \phi_{\Delta}$. From this, we can deduce that for all $\alpha_0 \in [\alpha, \alpha + \frac{2}{\log(1/\Delta)}]$ and for any estimator $\hat \phi_{\Delta}$, the following holds:
\begin{align*}
\mathbb{E} \|\phi_{\hat\alpha(\hat\phi_{\Delta}),\Delta} - \phi_{\alpha_0,\Delta}\|^{2} 
& \leq 2 \mathbb{E} \|\phi_{\hat\alpha(\hat\phi_{\Delta}),\Delta} - \hat\phi_{\Delta}\|^{2} + 2 \mathbb{E} \|\hat \phi_{\Delta} - \phi_{\alpha_0,\Delta}\|^{2} \\
& \leq 4 \mathbb{E} \|\hat\phi_{\Delta} - \phi_{\alpha_0,\Delta}\|^{2}.
\end{align*}

Substituting this into the previous inequality, we obtain
\begin{align*}
\inf_{\hat f_{\Delta}} \sup_{f_{\Delta} \in \mathcal S_{\alpha}} \mathbb{E} \|\hat f_{\Delta} - f_{\Delta}\|^{2} 
& \geq \frac{1}{8\pi} \inf_{\hat \alpha \in [0,2]} \sup_{\alpha_0 \in [\alpha, \alpha + \frac{2}{\log(1/\Delta)}]} \mathbb{E} \|\phi_{\hat\alpha,\Delta} - \phi_{\alpha_0,\Delta}\|^{2}.
\end{align*}

\emph{Step 3:} In this step, we control the quantity $\E\|\phi_{\hat\alpha,\Delta}-\phi_{\alpha_0,\Delta}\|^2$. 
For that we introduce the event $$A_{\hat\alpha,\alpha}=\Big\{\omega\in\Omega:  \log (1/\Delta) |\hat \alpha(\omega)-\alpha|\leq 2\Big\}$$ and the quantities $$m_{\alpha_0}:=\min(\alpha_0,\hat \alpha)\ge 0, \quad M_{\alpha_0}:=\max(\alpha_0,\hat \alpha)\le 2.$$ The step is divided in three points, where we control $\E\|\phi_{\hat\alpha,\Delta}-\phi_{\alpha_0,\Delta}\|^2$ respectively on $A_{\hat\alpha,\alpha}$ and $A_{\hat\alpha,\alpha}^{c}$ before concluding.

\emph{Step 3.1:} We first restrict to the event  $A_{\hat\alpha,\alpha}$ and $\alpha_0\in [\alpha,\alpha+\frac{2}{\log(1/\Delta)}]$. By Lagrange theorem applied to the function $h(x)=e^{-\Delta u^x}$, for any $\omega \in A_{\hat\alpha,\alpha}$
 there exists $\xi\in[m_{\alpha_0}(\omega), M_{\alpha_0}(\omega)]$ such that
\begin{align*}
\Big(e^{-\Delta u^{m_{\alpha_0}(\omega)}}-e^{-\Delta u^{M_{\alpha_0}(\omega)}}\Big)^2&=|\hat \alpha(\omega)-\alpha_0|^2 \Delta^2 e^{-2\Delta u^\xi} u^{2\xi} \log^2(u)\\
&\geq \frac{|\hat \alpha(\omega)-\alpha_0|^2 \Delta^2}{M_{\alpha_0}^2(\omega)} \log^2(\Delta) e^{-2\Delta u^{M_{\alpha_0}(\omega)}} u^{2m_{\alpha_0}(\omega)},
\end{align*}
for all $u\geq \Delta^{-\frac{1}{M_{\alpha_0}(\omega)}}$. Using that $M_{\alpha_0}\leq 2$, $m_{\alpha_0}\geq 0$ and $\frac{2m_{\alpha_0}+1}{M_{\alpha_0}}\leq \frac{2{\alpha_0}+1}{{\alpha_0}}$ we derive

\begin{align}
\E\big[\|\phi_{\hat\alpha,\Delta}-&\phi_{\alpha_0,\Delta}\|^2\mathbf{1}_{A_{\hat\alpha,\alpha}}\big]= 2\int_0^\infty \E\bigg[\mathbf{1}_{A_{\hat\alpha,\alpha}} \Big(e^{-\Delta u^{m_{\alpha_0}}}-e^{-\Delta u^{M_{\alpha_0}}}\Big)^2\bigg]du\nonumber\\
&\geq 2\E\bigg[\mathbf{1}_{A_{\hat\alpha,\alpha}}\int_\frac{1}{ \Delta^{1/{M_{\alpha_0}}}}^\infty  \frac{|\hat \alpha-\alpha_0|^2 \Delta^2}{4} \log^2(\Delta) e^{-2\Delta u^{M_{\alpha_0}}} u^{2m_{\alpha_0}}du\bigg]\nonumber\\
&\geq \frac{1}{2} \log^2(\Delta) \E\bigg[\mathbf{1}_{A_{\hat\alpha,\alpha}}|\hat \alpha-\alpha_0|^2 \Delta^2 \frac{\Gamma\Big(\frac{2m_{\alpha_0}+1}{M_{\alpha_0}},2\Big)}{M_{\alpha_0}}\big(2\Delta\big)^{-\frac{2m_\alpha+1}{M_\alpha}} \bigg]\nonumber\\
&\geq \frac{2^{-\frac{2\alpha_0+1}{\alpha_0}}}{4}\log^2(\Delta)\E\bigg[\mathbf{1}_{A_{\hat\alpha,\alpha}}|\hat \alpha-\alpha_0|^2 \Delta^{\frac{2(M_{\alpha_0}-m_{\alpha_0})-1}{M_{\alpha_0}}} \Gamma\Big(\frac{2m_{\alpha_0}+1}{M_{\alpha_0}},2\Big) \bigg]\nonumber\\
 &\geq  \frac{2^{-\frac{2{\alpha_0}+1}{\alpha_0}}\int_2^\infty \frac{e^{-t}}{\sqrt t}dt}{4}\log^2(\Delta)\E\bigg[\mathbf{1}_{A_{\hat\alpha,\alpha}}|\hat \alpha-\alpha_0|^2 \Delta^{\frac{2(M_{\alpha_0}-m_{\alpha_0})-1}{M_{\alpha_0}}} \bigg]\label{eq:LB1}.
\end{align}
We write $\Delta^{\frac{2(M_{\alpha_0}-m_{\alpha_0})-1}{M_{\alpha_0}}}=\Delta^{\frac{2(M_{\alpha_0}-m_{\alpha_0})}{M_{\alpha_0}}}\Delta^{-\frac{1}{\alpha}+\frac{M_{\alpha_0}(\omega)-\alpha}{\alpha M_{\alpha_0}(\omega)}}$ and we notice that for any $\omega\in A_{\hat\alpha,\alpha}$ and $\alpha_0\in [\alpha,\alpha+\frac{2}{\log(1/\Delta)}]$, it holds:
\begin{align*}
\frac{M_{\alpha_0}(\omega)-\alpha}{\alpha_0M_{\alpha_0}(\omega)}&\leq \frac{|\hat \alpha(\omega)-\alpha|+|\alpha-\alpha|}{\alpha\alpha_0}\leq \frac{4}{\alpha^2\log(1/\Delta)},\\
\frac{ M_{\alpha_0}(\omega)-m_{\alpha_0}(\omega)}{M_{\alpha_0}(\omega)}&=\frac{|\hat \alpha(\omega)-\alpha_0|}{M_{\alpha_0}(\omega)}\leq \frac{4}{\alpha\log(1/\Delta)}.
\end{align*}
Hence, from the previous inequalities, we derive
\begin{align}\label{eqLBd}
\Delta^{\frac{2(M_{\alpha_0}-m_{\alpha_0})-1}{M_{\alpha_0}}}\geq  \Delta^{-\frac{1}{\alpha}}  \Delta^{\frac{4}{\alpha\log(1/\Delta)}\big(2+\frac{1}{\alpha}\big)}= \Delta^{-\frac{1}{\alpha}} e^{-\frac{4}{\alpha}\big(2+\frac{1}{\alpha_0}\big)}.
\end{align}
Moreover, we observe that  
\begin{align}\label{sup}
\sup_{\alpha_0\in[\alpha,\alpha+ \frac{2}{\log(1/\Delta)}]} 2^{-\frac{2\alpha_0+1}{\alpha}}\geq 2^{-\big(\frac{1}{\alpha}+2\big)}.
\end{align}
Injecting \eqref{eqLBd} and \eqref{sup} in \eqref{eq:LB1} we deduce that
\begin{align*}
\sup_{\alpha_0\in [\alpha,\alpha+\frac{2}{\log(1/\Delta)}]} \hspace{-0,3cm}\E\big[\|\phi_{\hat\alpha,\Delta}-\phi_{\alpha_0,\Delta}\|^2\mathbf{1}_{A_{\hat\alpha,\alpha}}\big]\geq \frac{C_0}{n\Delta^{\frac{1}{\alpha}}}\sup_{\alpha_0\in [\alpha,\alpha+\frac{2}{\log(1/\Delta)}]} \hspace{-0,3cm}\E\big[\mathbf{1}_{A_{\hat\alpha,\alpha}}n \log^2(\Delta)|\hat \alpha-\alpha_0|^2\Big],
\end{align*}
where $C_0$ is a positive constant only depending on $\alpha$ defined as follows:
\begin{align}\label{C0}
C_{0}=\frac{2^{-\frac{1}{\alpha}} e^{-\frac{4}{\alpha}\big(2+\frac{1}{\alpha}\big)}\int_2^\infty \frac{e^{-t}}{\sqrt t}dt}{16}.
\end{align}

\emph{Step 3.2:} We now restrict to the event  $A^{c}_{\hat\alpha,\alpha}$ and $\alpha_0\in [\alpha,\alpha+\frac{2}{\log(1/\Delta)}]$.
Let $\omega\in A^c_{\hat \alpha,\alpha}$ and $u\geq \Delta^{-\frac{1}{m_{\alpha_0}(\omega)}}$, we justify that $e^{-\Delta u^{m_{\alpha_0}(\omega)}}-e^{-\Delta u^{M_{\alpha_0}(\omega)}}\geq \frac{e^{-\Delta u^{m_{\alpha_0}(\omega)}}}{2}$. Indeed, using that $m_{\alpha_0}\le2 $ we  write
\begin{align*}
u^{M_{\alpha_0}(\omega)}-u^{m_{\alpha_0}(\omega)}&=u^{m_{\alpha_0}(\omega)}\Big(u^{|\hat \alpha(\omega)-\alpha_0|}-1\Big)\geq \frac{1}{\Delta}\bigg(\Delta^{\frac{2}{m_{\alpha_0}(\omega)\log(\Delta)}}-1\bigg)\\
&=\frac{e^{\frac{2}{m_{\alpha_0}(\omega)}}-1}{\Delta}\ge \frac{2}{\Delta m_{\alpha_0}(\omega)}\geq  \frac{1}{\Delta}\geq \frac{\log 2}{\Delta}, 
\end{align*}
which implies that $e^{\Delta(u^{M_{\alpha_0}(\omega)}-u^{m_{\alpha_0}(\omega)})}\geq 2$ and the desired inequality. As a consequence, we get 
\begin{align}\label{c} 
 \E\big[\|\phi_{\hat\alpha,\Delta}-\phi_{\alpha_0,\Delta}\|^2\mathbf{1}_{A^c_{\hat \alpha,\alpha}}\big]&\geq \frac{1}{2}\E\bigg[\mathbf{1}_{A^c_{\hat \alpha,\alpha}}\int_{\Delta^{-1/m_{\alpha_0}}}^{2\Delta^{-1/m_{\alpha_0}}}e^{-2\Delta u^{m_{\alpha_0}}}du\bigg]\nonumber\\
 &\geq \frac{1}{2}\E\bigg[\mathbf{1}_{A^c_{\hat \alpha,\alpha}}\Delta^{-\frac1{m_{\alpha_0}}}e^{-2^{1+m_{\alpha_0}}}\bigg]\geq \frac{1}{2}\E\bigg[\mathbf{1}_{A^c_{\hat \alpha,\alpha}}\Delta^{-\frac1{\alpha}}\Delta^{\frac{m_{\alpha_0}-\alpha}{\alpha m_{\alpha_0}}}e^{-8}\bigg].
\end{align} 
In order to control the term $\Delta^{\frac{m_{\alpha_0}-\alpha}{\alpha m_{\alpha_0}}}$ we notice that for $\alpha_0\in [\alpha,\alpha-\frac{2}{\log(\Delta)}]$ and $\omega\in A^c_{\hat \alpha,\alpha}$ it holds
$$
\Delta^\frac{m_{\alpha_0}-\alpha}{\alpha m_{\alpha_0}}\geq \begin{cases}
1 &\text{ if }\quad  m_{\alpha_0}(\omega)-\alpha\leq 0,\\
\Delta^{\frac{2}{\log(1/\Delta)\alpha^2}}=e^{-\frac{2}{\alpha^2}}&\text{ otherwise},
\end{cases}
$$
that is
\begin{align}\label{ma}
\Delta^\frac{m_{\alpha_0}-\alpha}{\alpha m_{\alpha_0}}\geq e^{-\frac{2}{\alpha^2}}.
\end{align}
Injecting \eqref{ma} in \eqref{c} we derive
$$\sup_{\alpha_0\in [\alpha,\alpha+\frac{2}{\log(1/\Delta)}]} \E\big[\|\phi_{\hat\alpha,\Delta}-\phi_{\alpha_0,\Delta}\|^2\mathbf{1}_{A^c_{\hat \alpha,\alpha}}\big]\geq \frac{e^{-8-\frac{2}{\alpha^2}}}{2}\Delta^{-1/\alpha}\E\big[\mathbf{1}_{A^c_{\hat \alpha,\alpha}}\big].$$

\emph{Step 3.3:} This step aims to show that
\begin{align}\label{eq:inf}
\inf_{\hat \alpha\in[0,2]}\sup_{\alpha_0\in [\alpha,\alpha+\frac{2}{\log(1/\Delta)}]} \hspace{-0.3cm}&\E\big[\|\phi_{\hat\alpha,\Delta}-\phi_{\alpha_0,\Delta}\|^2\big] \\
&\geq 
\inf_{\hat \alpha\in[0,2]}\sup_{\alpha_0\in [\alpha,\alpha+\frac{2}{\log(1/\Delta)}]}\frac{K_0}{n\Delta^{1/\alpha}} \E[n \log^2(\Delta)|\hat \alpha-\alpha_0|^2], \nonumber
\end{align}
for some $K_0>0$ which depends only on $\alpha$.
Denote by
$$h(\hat\alpha,\alpha):= \sup_{\alpha_0\in [\alpha,\alpha+\frac{2}{\log(1/\Delta)}]} \E\Big[\mathbf{1}_{A_{\hat \alpha,\alpha}}n \log^2(\Delta)|\hat \alpha-\alpha_0|^2+n\mathbf{1}_{A^c_{\hat \alpha,\alpha}}\Big].$$
Steps 3.1 and 3.2 allows to write that
\begin{align*}
\inf_{\hat \alpha\in[0,2]}\sup_{\alpha_0\in [\alpha,\alpha+\frac{2}{\log(1/\Delta)}]} \E\big[\|\phi_{\hat\alpha,\Delta}-\phi_{\alpha_0,\Delta}\|^2\big]&\geq \frac{K_0}{n\Delta^{1/\alpha}} \inf_{\hat \alpha} h(\hat\alpha,\alpha),
\end{align*}
for $K_0=\min\big(\frac{e^{-8-{2}/{\alpha^2}}}{2}, C_0\big)$, with $C_0$ defined  in \eqref{C0}.
To show that \eqref{eq:inf} holds true, we show that 
$\inf_{\hat \alpha}h(\hat\alpha,\alpha)=\inf_{\hat \alpha: A_{\hat \alpha,\alpha}=\Omega}h(\hat\alpha,\alpha).$
Clearly, $\inf_{\hat \alpha}h(\hat\alpha,\alpha)\leq\inf_{\hat \alpha: A_{\hat \alpha,\alpha}=\Omega}h(\hat\alpha,\alpha)$. 
To prove the other inequality, we only need to prove that, given an arbitrary estimator $\hat \alpha$ of $\alpha_0$, we can construct $\check \alpha$ such that $A_{\check \alpha,\alpha}=\Omega$ and $h(\check \alpha,\alpha) \leq h(\hat \alpha,\alpha)$. Define:
$$\check \alpha(\omega):=
\begin{cases}
\hat \alpha(\omega),\quad &\text{if } \omega\in A_{\hat\alpha,\alpha},\\
\alpha-\frac{1}{\log(\Delta)},\quad &\text{if } \omega\in A^c_{\hat\alpha,\alpha}.
\end{cases} 
$$ It follows that
$$h(\check \alpha,\alpha)=\sup_{\alpha_0\in [\alpha,\alpha+\frac{2}{\log(1/\Delta)}]}\E[n\log^2(\Delta)|\check \alpha-\alpha_0|^2],$$
 by definition of $\check \alpha$, $A_{\check \alpha,\alpha}=\Omega$. On the other hand, $h(\check \alpha,\alpha)\leq h(\hat \alpha,\alpha)$ as if $\omega\in A_{\hat \alpha,\alpha}$ then $\hat \alpha(\omega)=\check \alpha(\omega)$ and if $\omega\in A^c_{\hat \alpha, \alpha}$, then 
\begin{align*}
h(\check \alpha,\alpha)&=\sup_{\alpha_0\in[\alpha,\alpha+\frac{2}{\log(1/\Delta)}]}n\log^2(\Delta)\Big|\alpha-\frac{1}{\log(\Delta)}-\alpha_0\Big|^2\leq n=h(\hat \alpha,\alpha).
\end{align*}

\emph{
Step 4:} The last step consists in proving that 
\begin{align}\label{fineq}
\inf_{\hat \alpha\in[0,2]}\sup_{\alpha_0\in[\alpha,\alpha+\frac{2}{\log(1/\Delta)}]}\frac{1}{n\Delta^{1/\alpha}} \E[n \log^2(\Delta)|\hat \alpha-\alpha_0|^2]\geq \frac{c}{n\Delta^{1/\alpha}},
\end{align}
for some positive constant $c$ and for $n\log^2(\Delta)$ large enough.
To that aim, we use the general theory on minimax bounds for parametric estimators under the LAN condition developed in Ibragimov and Has'minskii  \cite{ibragimov2013statistical}. Indeed, Masuda establishes in
\cite{belomestny2015levy} Theorem 3.2 p.218 that the family of symmetric $\alpha$-stable processes satisfy the LAN condition with a Fisher information given by $r_{n,\Delta,\alpha}:=I_{\alpha} n\log^{2}(1/\Delta),$ where $I_{\alpha}$ is a finite positive constant only depending on the parameter $\alpha\in (0,2)$ (see Equations (3.5) and (3.6) therein). 
Equation \eqref{fineq} is a consequence of Theorem 12.1 p.162 in \cite{ibragimov2013statistical}, which ensures that for $n\log^2(\Delta)$ large enough, for any $\delta>0$
$$ \inf_{\hat \alpha\in[0,2]} \sup_{|\alpha-\alpha_0|\leq \delta} \E[r_{n,\Delta,\alpha}|\hat\alpha-\alpha_0|^2]\geq \tilde c>0,$$
for some positive constant $\tilde c$. In our case, the interval over which the supremum is taken is not symmetric but rather of the form $[\alpha, \alpha-\frac{2}{\log(\Delta)}]$. Adapting the proof of the Theorem 12.1, the same conclusion holds true, for a different constant. The ingredient that ensures the validity of the proof is the fact that $\frac{ r_{n,\Delta,\alpha}}{\delta^{2}}=I_{\alpha} n\to \infty$ as $n\to\infty$.
Finally, observing that for $\Delta\leq e^{-\frac{4}{2-\alpha}}$ we have $[\alpha, \alpha-\frac{2}{\log(\Delta)}] \subseteq [\alpha,\frac{\alpha}{2}+1]$, we derive 
\begin{align}\label{dirf}
&\inf_{\hat \alpha\in[0,2]}\sup_{\alpha_0\in [\alpha, \alpha-\frac{2}{\log(\Delta)}]}\frac{\E[n \log^2(\Delta)|\hat \alpha-\alpha_0|^2]}{n\Delta^{1/\alpha}} \nonumber\\
&\quad \geq \inf_{\alpha_0\in [\alpha,\frac{\alpha}{2}+1]}\frac{1}{I_{\alpha_0}}\inf_{\hat \alpha\in[0,2]} \sup_{\alpha_0\in [\alpha, \alpha-\frac{2}{\log(\Delta)}]} \E[r_{n,\Delta,\alpha}|\hat\alpha-\alpha_0|^2]\ge  \inf_{\alpha_0\in [\alpha,\frac{\alpha}{2}+1]}\frac{1}{I_{\alpha_0}}\tilde c.
\end{align}
Recalling that $I_{\alpha_0}$ is finite for any $\alpha_0\in(0,2)$, hence it is bounded on $[\alpha,\frac{\alpha}{2}+1]$, we derive \eqref{fineq} as desired. The proof is now complete.

\subsection{Proof of Theorem \ref{thm:adapt}}
Let $m\in(0,n]$ and note that 
\begin{align*}
{\mathbb E}[\|\tilde{f}_\Delta-f_\Delta\|^{2}] &\le \|f_{\Delta, m}-f_\Delta\|^{2}+{\mathbb E}[\|\tilde{f}_\Delta-f_{\Delta, m}\|^{2}]\\
&= \frac1{2\pi}\int_{[-m,m]^{c}}| \phi_{X_{\Delta}}(u)|^{2}{ d}u+{\mathbb E}[\|\tilde f_\Delta-f_{\Delta,m}\|^{2}].
\end{align*}
The first term is a bias term, we compute the second variance term using  Parseval's equality:\begin{align}2\pi{\mathbb E}[\|\tilde f_\Delta-f_{\Delta,m}\|^{2}]&=\int_{[-m,m]}{\mathbb E}\big[|\tilde \phi_{X_{\Delta}}(u)-\phi_{X_{\Delta}}(u)|^{2}\big]{ d}u +\int_{[-n,n]\setminus[-m,m]}|\tilde \phi_{X_{\Delta}}(u)|^{2}du.\label{eq:step1bis}\end{align} 
The first term is a variance term which, using the definition of $\tilde \phi_{X_{\Delta}}$, is controlled as follows:
\begin{align}
{\mathbb E}\big[|\tilde \phi_{X_{\Delta}}(u)-\phi_{X_{\Delta}}(u)|^{2}\big]&\le
\label{eq:var1}
{\mathbb E}\big[|\hat \phi_{X_{\Delta}}(u)-\phi_{X_{\Delta}}(u)|^{2}\big]+|\phi_{X_{\Delta}}(u)|^{2}{\mathbb P}\Big(|\hat\phi_{X_{\Delta}}(u)|< \frac{\kappa_{n}}{\sqrt{n}}\Big).
\end{align}
The first term in the right hand side of \eqref{eq:var1} is bounded by ${1}/{n}$.
For the second term, as $\kappa_{n}=1+\kappa\sqrt{\log n}$, consider the set $\{u,|\phi_{X_{\Delta}}(u)|< \frac{1+(\kappa+2)\sqrt{\log n}}{\sqrt{n}}\}$ and its complementary set where we can write from the triangle inequality that 
\begin{align*}
\bigg\{u,\ |\hat\phi_{X_{\Delta}}(u)|< \frac{1+\kappa\sqrt{\log n}}{\sqrt{n}}, \ |&\phi_{X_{\Delta}}(u)|> \frac{1+(\kappa+2)\sqrt{\log n}}{\sqrt{n}}\bigg\}\\
& \subset \left\{u,\ |\hat\phi_{X_{\Delta}}(u)-\phi_{X_{\Delta}}(u)|\geq \frac{2\sqrt{\log n}}{\sqrt{n}}\right\}.
\end{align*}
This leads to
\begin{align}\label{eq:var2}
\hspace{-0.5cm}|\phi_{X_{\Delta}}(u)|^{2}{\mathbb P}\Big(|\hat\phi_{X_{\Delta}}(u)|< \frac{\kappa_{n}}{\sqrt{n}}\Big)&
\leq \frac{(1+(\kappa+2)\sqrt{\log n})^{2}}{n} \nonumber \\
&\phantom{=} +{\mathbb P}\Big(|\hat\phi_{X_{\Delta}}(u)-\phi_{X_{\Delta}}(u)|\geq \frac{2\sqrt{\log n}}{\sqrt{n}}\Big) \nonumber \\
&\leq  \frac{(1+(\kappa+2)\sqrt{\log n})^{2}}{n}+\frac 4 n \leq  \frac{4+(1+(\kappa+2)\sqrt{\log n})^{2}}{n}, 
\end{align} where we used  the Hoeffding inequality on the bounded complex variables $e^{iu(X_{j\Delta}-X_{(j-1)\Delta})}$, see Lemma in \cite{ammous2024adaptive} with $b=2$. It follows that 
\[\int_{[-m,m]}{\mathbb E}\big[|\tilde \phi_{X_{\Delta}}(u)-\phi_{X_{\Delta}}(u)|^{2}\big]{ d}u \le 2m \frac{5+(1+(\kappa+2)\sqrt{\log n})^{2}}{n}.\]
Now, consider the second term in \eqref{eq:step1bis}. Use that
$
|\tilde\phi_{X_{\Delta}}|^2   \leq 2 |{\phi}_{X_{\Delta}}|^2
+ 2 | \tilde{\phi}_{X_{\Delta}}-\phi_{X_{\Delta}}|^2 
$ a.s.,  to write
\begin{align*}
\int_{[-n,n]\setminus [-m,m]}\hspace{-0.6cm}{\mathbb E}&\big[|\tilde \phi_{X_{\Delta}}(u)|^{2}\big]{ d}u \\
& \le 2\int_{[-n,n]\setminus [-m,m]}\hspace{-0.6cm}| \phi_{X_{\Delta}}(u)|^{2}{ d}u+2\int_{[-n,n]\setminus [-m,m]}\hspace{-0.6cm}{\mathbb E}\big[|\tilde \phi_{X_{\Delta}}(u)-\phi_{X_{\Delta}}(u)|^{2}\big]{ d}u\\
& \le 2\int_{[-m,m]^{c}}| \phi_{X_{\Delta}}(u)|^{2}{d}u+2\int_{[-n,n]\setminus [-m,m]}{\mathbb E}\big[|\tilde \phi_{X_{\Delta}}(u)-\phi_{X_{\Delta}}(u)|^{2}\big]{ d}u.
\end{align*} The first term is a bias term. Consider the set
 $\Phi:=\{u, |\phi_{X_{\Delta}}(u)|>n^{-1/2}\}$ to handle the second term. For any $u\in\Phi$, we recover a bias term as it holds: \begin{align*}
 {\mathbb E}[|\tilde{\phi}_{X_{\Delta}}(u)- \phi_{X_{\Delta}}(u)|^2]&\leq|\phi_{X_{\Delta}}(u)|^{2}+ {\mathbb E}[|\hat{\phi}_{X_{\Delta}}(u)- \phi_{X_{\Delta}}(u)|^2]\\
 &\leq|\phi_{X_{\Delta}}(u)|^{2}+\frac1n\leq 2|\phi_{X_{\Delta}}(u)|^{2}.
 \end{align*}
On the complementary $ \Phi^{c}$, we use
\begin{align*}
{\mathbb E}  \Big[   |\tilde{\phi}_{X_{\Delta}}(u)- &\phi_{X_{\Delta}}(u)|^2 \mathbf{1}_{\{|\phi_X(u)| \leq  n^{-1/2}\}}\Big] \\
&= {\mathbb E}  \Big[   | {\phi}_{X_{\Delta}}(u)|^2 \mathbf{1}_{\{|{\phi}_{X_{\Delta}}(u)| \leq  n^{-1/2}, |\hat{\phi}_{X_{\Delta}}(u)|<\kappa_nn^{-1/2}\}}\Big]\\
&\phantom{=}+{\mathbb E}  \Big[   |\hat{\phi}_{X_{\Delta}}(u)-{\phi}_{X_{\Delta}}(u)|^2 \mathbf{1}_{\{|{\phi}_{X_{\Delta}}(u)| \leq  n^{-1/2},  |\hat{\phi}_{X_{\Delta}}(u)|\ge\kappa_nn^{-1/2}\}}\Big]\\
&\le  | {\phi}_{X_{\Delta}}(u)|^2+ 4\mathbb P\big(|{\phi}_{X_{\Delta}}(u)| \leq  n^{-1/2},  |\hat{\phi}_{X_{\Delta}}(u)|\ge\kappa_nn^{-1/2}\big),
\end{align*}
where we used that $|\hat{\phi}_{X_{\Delta}}(u)- {\phi}_{X_{\Delta}}(u)|^2\leq 4$. Next, using that $\kappa_{n}=1+\kappa\sqrt{\log n}$, we get $$\{u,\ |{\phi}_{X_{\Delta}}(u)| \leq  n^{-1/2}\le\kappa_{n}n^{-1/2}\le  |\hat{\phi}_{X_{\Delta}}(u)| \}\subset \{u,\ \kappa\sqrt{\log n}n^{-1/2}\le   |\hat{\phi}_{X_{\Delta}}(u) - {\phi}_{X_{\Delta}}(u)|\}.$$
 It follows that
\begin{align*}
{\mathbb E}  \Big[   |\tilde{\phi}_{X_{\Delta}}(u)- &{\phi}_{X_{\Delta}}(u)|^2 \mathbf{1}_{\{|{\phi}_{X_{\Delta}}(u)| \leq  n^{-1/2}\}}\Big]\\
&\leq  | {\phi}_{X_{\Delta}}(u)|^2+ 4\mathbb P\big(|\hat{\phi}_{X_{\Delta}}(u)-{\phi}_{X_{\Delta}}(u)|\ge\kappa(\log n/n)^{1/2}\big).
\end{align*}
 This allows to write that
  \begin{align*}
&{\mathbb E}  \Big[ \int_{[-n,n]\setminus [-m,m]} \limits\hspace{-0.5cm}  |\tilde{\phi}_{X_{\Delta}}(u)- \phi_{X_{\Delta}}(u)|^2 \mathbf{1}_{\{|\phi_{X_{\Delta}}(u)| \leq  n^{-1/2}\}}{d}u\Big] \nonumber\\
&\leq \int_{[-m,m]^{c} }\limits    |\phi_{X_{\Delta}}(u)|^2 {d}u + 4 \int_{[-n,n]}\limits {\mathbb P}( |\hat{\phi}_{X_{\Delta}}(u) - \phi_{X_{\Delta}}(u)| > \kappa (\log n/n)^{1/2}  )   { d}u  
\nonumber\\ & \leq\int_{[-m,m]^{c} }\limits   |\phi_{X_{\Delta}}(u)|^2 {d}u +  32  n^{ 1 - \kappa^2/4}, 
 \end{align*}
  where the last inequality is a consequence of  Lemma in \cite{ammous2024adaptive} with $b=\kappa$. 
 Gathering all terms we obtain
 that  for all $m\in(0,n]$,
\begin{align*}
{\mathbb E}[  \| \tilde{f}_\Delta- f_{\Delta}\|^2 ]&  \leq 9 \|f_{\Delta,m}-f_{\Delta}\|^{2} +\frac m{\pi n} (5+(1+(\kappa+2)\sqrt{\log n})^{2}) +64n^{1-\frac{\kappa^{2}}{4}}.\end{align*}
Taking the infimum over $m$ completes the proof.

\subsection{Proof of Theorem \ref{thm:rate_sigma}}
 If $(b,\sigma^2,\nu)\in\mathcal T_{+,M,\alpha}$, we use the following bound on the bias term. Set $c_\alpha = 2M({2}/{\pi})^{\alpha} $, it holds:
$$\|f_{\Delta,m}-f_{\Delta}\|^{2}\leq  \frac{1}{\pi} \int_{m}^\infty e^{-\Delta \sigma^2 u^2 - c_\alpha \Delta u^\alpha}du=:b^{2}(m),\quad \forall \Delta>0,\ m\geq \frac{\pi}{2}.
$$
Recall that $V(m)=m/(\pi n)$ is the variance term.  To find $m^{\star}$ minimizing $m\mapsto b^{2}(m)+V(m)$, we  differentiate this function and solve
\begin{align}\label{eq:m_star_eq_brw}
\sigma_n^2 \Delta {m^\star}^2 + c_\alpha\Delta {m^\star}^\alpha=c_\alpha \Delta {m^\star}^\alpha \left(1+ \frac{\sigma_n^2 {m^\star}^{2-\alpha}}{c_\alpha}\right) =\log n.
\end{align}
Note that $m^\star$ then depends on $n$ and is such that $m^\star = m^\star_n \underset{n\rightarrow \infty}{\rightarrow} \infty$. Its value varies according to which term dominates between $\sigma_n^2 {m_{n}^\star}^2$ and ${m_{n}^\star}^\alpha$. 
\begin{enumerate}
\item 
If $\sigma_n^2 {m_{n}^\star}^{2-\alpha} \underset{n\rightarrow \infty}{\rightarrow} \infty$, select $m_{n}^\star =O\left(\sqrt{{\log n}/({\Delta \sigma_n^2})}\right)$. It satisfies the latter constraint when $\sigma_n (\log n)^{1/\alpha - 1/2} \underset{n\rightarrow \infty}{\rightarrow} \infty.$
\item If $\sigma_n^2 {m_{n}^\star}^{2-\alpha} \underset{n\rightarrow \infty}=O(1),$ meaning that it remains bounded away from $\infty$ and may either converge to a positive or null limit,  select $m_{n}^\star \underset{n\rightarrow \infty}=O\left({(\log n)^{1/\alpha}}{\Delta^{-1/\alpha}}\right)$. It satisfies the latter constraint when $
\sigma_n (\log n)^{ 1/\alpha-1/2 } \underset{n\rightarrow \infty}=O(1).
$
\end{enumerate}Note that the transition from one regime to the other is continuous. In both cases, we 
compare the order of the square bias term and the variance term to determine which of the two drives the rate. Using that, for $m\ge \pi/2$,
\begin{align}\label{eq:bias_control}
\pi^{-1} b^2(m)&\leq e^{-\Delta \sigma_n^2 m^2} \int_{m}^\infty e^{-c_\alpha \Delta u^\alpha}du=\frac1{ \alpha c_\alpha^{1/\alpha}\Delta^{1/\alpha}}e^{-\Delta \sigma_n^2 m^2}  \Gamma\left(1/\alpha,c_\alpha \Delta m^\alpha \right),
\end{align}
where $\Gamma(.,.)$ is the upper incomplete Gamma function, together with
 $\Gamma(s,x) \underset{x\rightarrow \infty}{\sim} x^{s-1}e^{-x}$, we obtain the following upper bound of the right hand side of \eqref{eq:bias_control}
\begin{align*}
\pi^{-1} b^2(m) &\le \frac C\Delta e^{-\Delta \sigma_n^2 m^2 - c_\alpha \Delta m^\alpha} m^{1-\alpha },
\end{align*} for a positive constant $C$ depending on $\alpha$. This inequality taken at $m^{\star }_{n}$ solution of  \eqref{eq:m_star_eq_brw} gives
\begin{align*}
 b^2(m^{\star}_{n}) &\le \frac C\Delta\frac{{m^{\star}_{n}}^{1-\alpha }}{\pi n}=\frac C\Delta{V(m^{\star}_{n}) {m^{\star}_{n}}^{-\alpha }}.
\end{align*}
Therefore, the rate  is variance dominated and the announced rates of convergence follow.

\begin{appendix}
\section{Appendix}
\subsection{ A technical Lemma} The result below directly follows from Lemma 2.3 in Picard \cite{picard1997}.
\begin{lemma}\label{lemmapicard}
Let $X$ be a Lévy process with Lévy density $p$ satisfying \eqref{Ass:p}. 
Then, denoting by $\phi_t(u)=\E[e^{iuX_t}]$, it holds:
\begin{align}\label{eq:Picard1}
|\phi_t(u)|\leq e^{-\frac{2^{\alpha}M}{\pi^{\alpha}}|u|^{\alpha}t},\quad \forall |u|\geq \frac{\pi}{2},\ \forall t>0.
\end{align}
Furthermore, $X_t$ has a smooth density $f_t$ with all its derivates uniformly bounded:
\begin{align}\label{eq:Picard2}
\sup_{x\in\R} |f_t^{(k)}(x)|\leq\frac1{\pi(k+1)}\left(\frac{\pi}{2}\right)^{k+1}+ \frac1\alpha\left(\frac{\pi}{2(tM)^{\frac1\alpha}}\right)^{k+1}\Gamma\left(\frac{k+1}{\alpha},{tM}\right),\quad \forall k\geq 0, \ \forall t>0,
\end{align}
where $\Gamma(a,x)=\int_{x}^{\infty}z^{a-1}e^{-z}dx$ denotes the incomplete Gamma function, $a,x>0$.
\end{lemma}

\begin{proof}[Proof of Lemma \ref{lemmapicard}]
We follow the lines of Lemma 2.3 in Picard \cite{picard1997}. The Lévy-Khintchine formula allows to write for $u\ne 0$
\begin{align*}
|\phi_{t}(u)|=\exp\left(t\int_{\R}(\cos(u x)-1)\nu(dx)\right)\le\exp\left(t\int_{-\frac{\pi}{2|u|}}^{\frac{\pi}{2|u|}}(\cos(u x)-1)\nu(dx)\right).
\end{align*} Using that for $|x|\le {\pi}/{2}$ it holds $-{x^{2}}/{2}\le \cos(x)-1\le -{4x^{2}}{\pi^{-2}}$ and \eqref{Ass:p}, we deduce that for all $u,\ |u|\ge {\pi}/{2}$, 
$$ |\phi_{t}(u)|\le\exp\left(-\frac{4tu^{2}}{\pi^{2}}\int_{-\frac{\pi}{2|u|}}^{\frac{\pi}{2|u|}}x^{2}p(x)dx\right)\le e^{-\frac{2^{\alpha}M}{\pi^{\alpha}}|u|^{\alpha}t}.$$ In particular, this ensures that $u\mapsto u^{k}|\phi_{t}(u)|$ is integrable for any $k\in\N$, that $f_{t}$ is in $C^{k}(\R)$ and for all $x\in\R$ it holds that
\begin{align*}
|f_{t}^{(k)}(x)|&\le \frac1\pi\int_{0}^{\infty}u^{k}|\phi_{t}(u)|du\le\frac1{\pi(k+1)}\left(\frac{\pi}{2}\right)^{k+1}+ \frac1\alpha\left(\frac{\pi}{2(tM)^{\frac1\alpha}}\right)^{k+1}\Gamma\left(\frac{k+1}{\alpha},{tM}{}\right),
\end{align*} where the integral is  split  at $\pi/2$.

\end{proof}

\end{appendix}

\bibliographystyle{imsart-number} 
\bibliography{refs}    

\end{document}